\documentclass[11pt]{amsart}
\usepackage{amssymb}
\usepackage{mathtools}		
\usepackage{mathabx}		
\usepackage{mathrsfs}		

\usepackage[normalem]{ulem} 

\usepackage{amsthm}
\usepackage{thmtools}

\usepackage{enumitem}		
\usepackage[backref=page,colorlinks=true]{hyperref}	
\usepackage[capitalize]{cleveref} 

\usepackage[font=small]{caption}

\usepackage{tikz}

\usepackage{xcolor}
\hypersetup{
    colorlinks,
    linkcolor={red!50!black},
    citecolor={blue!50!black},
    urlcolor={blue!80!black}
}

\renewcommand*{\backref}[1]{}
\renewcommand*{\backrefalt}[4]{\quad \tiny 
  \ifcase #1 (\textbf{NOT CITED.})%
  \or    (Cited on page~#2.)%
  \else   (Cited on pages~#2.)%
  \fi}

\makeatletter

\def\MRbibitem{\@ifnextchar[\my@lbibitem\my@bibitem}

\def\mybiblabel#1#2{\@biblabel{{\hyperref{http://www.ams.org/mathscinet-getitem?mr=#1}{}{}{#2}}}}

\def\myhyperanchor#1{\Hy@raisedlink{\hyper@anchorstart{cite.#1}\hyper@anchorend}}

\def\my@lbibitem[#1]#2#3#4\par{%
  \item[\mybiblabel{#2}{#1}\myhyperanchor{#3}\hfill]#4%
  \@ifundefined{ifbackrefparscan}{}{\BR@backref{#3}}%
  \if@filesw{\let\protect\noexpand\immediate
    \write\@auxout{\string\bibcite{#3}{#1}}}\fi\ignorespaces%
}

\def\my@bibitem#1#2#3\par{%
  \refstepcounter\@listctr
  \item[\mybiblabel{#1}{\the\value\@listctr}\myhyperanchor{#2}\hfill]#3%
  \@ifundefined{ifbackrefparscan}{}{\BR@backref{#2}}%
  \if@filesw\immediate\write\@auxout
    {\string\bibcite{#2}{\the\value\@listctr}}\fi\ignorespaces%
}

\makeatother

\declaretheoremstyle[
headfont=\footnotesize\bf,
bodyfont=\footnotesize\sf
]{myremark}

\declaretheorem[Refname={Theorem,Theorems}]{theorem}

\declaretheorem[Refname={Theorem,Theorems}, name=Theorem, numberwithin=section]{otherthm}
\declaretheorem[Refname={Lemma,Lemmas}, sibling=otherthm]{lemma}
\declaretheorem[Refname={Corollary,Corollaries}, sibling=otherthm]{corollary}
\declaretheorem[Refname={Proposition,Propositions}, sibling=otherthm]{proposition}
\declaretheorem[Refname={Definition,Definitions}, sibling=otherthm, style=definition]{definition}
\declaretheorem[Refname={Example,Examples}, sibling=otherthm, style=remark]{example}
\declaretheorem[Refname={Question,Questions}, sibling=otherthm, style=remark]{question}
\declaretheorem[Refname={Remark,Remarks}, sibling=otherthm, style=remark]{remark}
\declaretheorem[Refname={Claim,Claims},sibling=otherthm, style=remark]{claim}

\hypersetup{bookmarksdepth = 3} 
\numberwithin{equation}{section}     

\setlist[enumerate,1]{label={\upshape(\alph*)},ref=\alph*}
\setlist[enumerate,2]{label={\upshape(\arabic*)},ref=\arabic*}

\newcommand{\tribar}[1]{\mathopen{| {\kern -1.5pt} | {\kern -1.5pt} |} {#1}
\mathclose{| {\kern -1.5pt} | {\kern -1.5pt} |}}

\newcommand{\qand}{\quad \text{and} \quad}
\newcommand{\A}{\mathbb{A}}
\newcommand{\R}{\mathbb{R}}

\newcommand{\Z}{\mathbb{Z}}

\newcommand{\cC}{\mathcal{C}}
\newcommand{\cD}{\mathcal{D}}\newcommand{\cF}{\mathcal{F}}
\newcommand{\cG}{\mathcal{G}}
\newcommand{\cL}{\mathcal{L}}
\newcommand{\cM}{\mathcal{M}}\newcommand{\cO}{\mathcal{O}}
\newcommand{\cR}{\mathcal{R}}
\newcommand{\cU}{\mathcal{U}}
\newcommand{\cV}{\mathcal{V}}\newcommand{\cW}{\mathcal{W}}

\newcommand{\st}{\;\mathord{;}\;}
\newcommand*\circled[1]{\tikz[baseline=(char.base)]{
    \node[shape=circle,draw,inner sep=1pt] (char) {\footnotesize{#1}};}}

\newcommand{\doi}[1]{\href{http://dx.doi.org/#1}{DOI:{#1}}}

\renewcommand{\phi}{\varphi}
\renewcommand{\setminus}{\smallsetminus}

\newcommand{\Leb}{\mathrm{Leb}}
\newcommand{\Diff}{\mathit{Diff}}
\newcommand{\id}{\mathrm{id}}
\DeclareMathOperator{\mo}{mo}
\DeclareMathOperator{\qo}{qo}
\DeclareMathOperator{\supp}{supp}
\DeclareMathOperator{\card}{\#}
\DeclareMathOperator{\diam}{diam}
\newcommand{\cercle}{\mathbb{T}} 
\newcommand{\Eme}{\mathscr{E}} 

\begin{document}
\title{On Emergence and Complexity of Ergodic Decompositions}
\date{January, 2019 (first version); June, 2021 (final version)}

\author[Pierre Berger]{Pierre Berger{}$^*$}
\address{Pierre Berger (\href{mailto:berger@math.univ-paris13.fr}{{\tt berger@math.univ-paris13.fr}}), Universit\'e Paris 13, Sorbonne Paris Cit\'e, \textsc{laga}, \textsc{cnrs} (\textsc{umr} 7539), F-93430, Villetaneuse, France.}
	
\author[Jairo Bochi]{Jairo Bochi{}$^{**}$}
\address{Jairo Bochi (\href{mailto:jairo.bochi@mat.uc.cl}{{\tt jairo.bochi@mat.uc.cl}}), Facultad de Matem\'aticas, 
	Pontificia Universidad Cat\'olica de Chile, 
	Avenida Vicu\~na Mackenna 4860
	Santiago, Chile}

\begin{abstract}
A concept of \emph{emergence} was recently introduced in \cite{Berger} in order to quantify the richness of possible statistical behaviors of orbits of a given dynamical system. In this paper, we develop this concept and provide several new definitions, results, and examples. We introduce the notion of \emph{topological emergence} of a dynamical system, which essentially evaluates how big the set of all its ergodic probability measures is. On the other hand, the \emph{metric emergence} of a particular reference measure (usually Lebesgue) quantifies how non-ergodic this measure is. We prove fundamental properties of these two emergences, relating them with classical concepts such as Kolmogorov's $\epsilon$-entropy of metric spaces and quantization of measures. We also relate the two types of emergences by means of a variational principle. Furthermore, we provide several examples of dynamics with high emergence. First, we show that the topological emergence of some standard classes of hyperbolic dynamical systems is essentially the maximal one allowed by the ambient. Secondly, we construct examples of  smooth area-preserving diffeomorphisms that are extremely non-ergodic in the sense that the metric emergence of the Lebesgue measure is essentially maximal. These examples confirm that super-polynomial emergence indeed exists, as conjectured in \cite{Berger}. Finally, we prove that such examples are locally generic among smooth diffeomorphisms. 
\end{abstract}

\begin{thanks}
{
{}$^*$ Partially supported by the ERC project 818737 \textit{Emergence of wild differentiable dynamical systems}. 818737  Emergence  of  wild  differentiable  dynamicalsystems.
{}$^{**}$ Partially supported by project Fondecyt 1180371 and project Conicyt PIA ACT172001 \textit{New trends in ergodic theory}.
}
\end{thanks}

\subjclass[2020]{37A35; 37C05, 37C45, 37C40, 37J40}

\keywords{Emergence, ergodic decomposition, Wassertein space, covering number, quantization number, dimension, entropy, KAM theorem.}

\maketitle

\belowpdfbookmark{\contentsname}{toc} 
\tableofcontents


\section*{Introduction}

An unsophisticated but fruitful way of quantifying the size of a compact metric space $X$ goes as follows:
one counts how many points can be distinguished up to error $\epsilon > 0$, and then studies the behavior of this number $N(\epsilon)$ as the resolution $\epsilon$ tends to zero. 
For example, if $N(\epsilon)$ is of the order of $\epsilon^{-d}$, for some $d>0$, then we say that $X$ has (box-counting) dimension $d$.
This dimension, when it exists, is a geometric invariant of $X$: it is preserved under bi-Lipschitz maps.

A similar idea can be used to define invariants of dynamical systems.
One considers how many orbits can be distinguished up to time $t>0$ and up to a fine resolution; if this number is roughly $\exp(h\cdot t)$ then the dynamics has topological entropy $h$. 
The metric entropy (also called Kolmogorov--Sinai entropy) of an invariant probability measure can be characterized similarly: in this case we are allowed to disregard a set of orbits of small probability.

This discretization paradigm can be used to quantify the complexity of a dynamical system in another way, called \emph{emergence}, which was recently introduced in \cite{Berger}.  
Emergence is significant when a finite number of statistics is not enough to describe the behavior of the orbit of almost every point. 
In this paper, we carry out a more detailed study of the concept of emergence, sometimes guided by analogies with the concept of entropy.
Furthermore, we provide examples of topologically generic dynamics with high emergence, substantiating a conjecture from \cite{Berger}. 

Let us note that the word \emph{emergence} is used with several different meanings in the scientific literature. Our use is compatible with MacKay's viewpoint, according to whom ``emergence means non-unique statistical
behaviour'' \cite{MacKay}. He elaborates on this as follows:
\begin{quote}
``Note that emergence is very different from
chaos, in which sensitive dependence
produces highly non-unique trajectories
according to their initial conditions. Indeed, the
nicest forms of chaos produce unique statistical
behaviors in the basin of the attractor. The
distinction is like that between the weather and
the climate. For weather we care about
individual realizations; for climate we care about
statistical averages.'' \cite{MacKay}
\end{quote}

\medskip

Given a continuous self-map $f$ of a compact metric space $X$, emergence distinguishes only the statistical behavior of orbits of $f$.
So it does not  matter \emph{when} a segment of orbit $(f^i(x))_{i=0}^{n-1}$ visits a certain region of the phase space $X$, but only \emph{how often}. This can be quantified by a probability, the \emph{$n^{\mathrm{th}}$ empirical measure associated to $x$}:
\[
\mathbf{e}_n^{f}(x) \coloneqq  \frac1n \sum_{i=0}^{n-1} \delta_{f^i(x)}  \; .
\] 

In the paradigm of ergodic theory, one focuses on the probability measures~$\mu$ which are \emph{invariant}: $f_*\mu=\mu$. We denote by $\cM_f(X)$ the convex, closed subset of such measures. Then, by Birkhoff ergodic theorem, for $\mu$-a.e.\ $x\in X$, the sequence $(\mathbf{e}_n^{f}(x))_n$ converges to a unique measure: 
\[
\mathbf{e}^{f}(x)\coloneqq \lim_{n\to \infty} \mathbf{e}_n^{f}(x)\; ,
\]
called the \emph{empirical measure associated to $x$}. 
Furthermore, this measure is  almost surely ergodic,  by the ergodic decomposition theorem. 
We recall that a measure $\mu$ is \emph{ergodic} if and only  the \emph{empirical function} $x\mapsto \mathbf{e}^f(x)$ is $\mu$-a.e.\ constant.
We denote by $\cM_f^\mathrm{erg}(X)\subset \cM_f(X)$ the subset of ergodic probability measures.

It is natural to study how many ergodic statistical behaviors a dynamical system admits up to resolution $\epsilon$ (in a sense to be made precise). We are interested in the behavior of this number as $\epsilon$ tends to zero. This leads us to introduce the following notion:

\begin{definition}[Topological Emergence] \label{def.te}
Let $X$ be a compact metric space, let $f$ be a continuous self-map of $X$, and let $\mathsf{d}$ be a distance on the space of probabilities $\cM(X)$ of $X$ so that $(\cM(X),\mathsf{d})$ is compact.

The \emph{topological emergence} $\Eme_\mathrm{top}(f)(\epsilon)$ of $f$ is the function which associates to $\epsilon>0$ the minimal number of $\epsilon$-balls of $\cM(X)$ whose union covers $\cM_f^\mathrm{erg}(X)$.
\end{definition}

Of course this definition depends on how the space of measures is metrized. 
There are basically two classical types of distances on the space of probabilities $\cM(X)$ which define the same weak topology (which is the most relevant one in ergodic theory): the L\'evy--Prokhorov distance $\mathsf{LP}$, and Wasserstein distances $\mathsf{W}_p$, which depend on a parameter $p \in [1,\infty)$.
We will recall their definitions in \cref{ss:distproba}. 
For the rest of this introduction, we fix any distance $\mathsf{d} \in \{\mathsf{W}_p: 1\le p<\infty\} \cup \{\mathsf{LP}\}$. 

In \cref{s:top_em_examples}, we will give examples of open sets of mappings with essentially maximal topological emergence:
 
\begin{theorem}\label{emergence:topo:example}
Let $f$ be $C^{1+\alpha}$-mapping of a manifold which admits a basic hyperbolic set $K$ with box-counting dimension $d$. Assume that  $f$ is conformal expanding or that $f$ is a conservative surface diffeomorphism. Then the topological emergence of $f|K$ is stretched exponential with exponent $d$:
\[\lim_{\epsilon\to 0} \frac{\log \log \Eme_\mathrm{top}(f|K) (\epsilon)}{-\log \epsilon}=d\; .\]
\end{theorem}

The emergence exponent is indeed maximal, since for any such a compact set $K$, the covering number 
of the space of probability measures $\cM(K)$ is stretched exponential with exponent $d$, both for the L\'evy--Prokhorov metric $\mathsf{LP}$ and the Wasserstein metrics $\mathsf{W}_p$; see \cref{section:mo:was} for details. 

The concept of topological emergence is linked to classical ideas of size of functional spaces developed by the Kolmogorov school (and emanating from Hilbert's 13th problem): see \cref{s:mo} for more details.
Let us also note that the set $\cM_f(X)$ has been investigated from several (topological, convex-analytic, \dots) points of view for various classes of maps $f$, from older works \cite{Sigmund,Dow} to recent ones \cite{GoP,GR,BBG,BZ,DGMR,DGR}.
The study of topological emergence expands this theme of research by imparting a more quantitative aspect to it.

\medskip

We may be only interested in physically relevant statistics, and so we are allowed to disregard statistics that correspond to a set of orbits of zero Lebesgue measure. This led to the following concept \cite{Berger}, initially introduced for $X$ a manifold and $\mu$ the Lebesgue measure:

\begin{definition}[Metric Emergence] \label{def.metric_em}
Let $(X,\mu)$ be a compact metric space $X$ endowed with a probability measure $\mu$, and let $f$ be a continuous self-map of $X$ (not necessarily $\mu$-preserving).

The \emph{metric emergence} $\Eme_{\mu}(f)$ is the function that associates to $\epsilon>0$ the minimal number $\Eme_{\mu}(f)(\epsilon)=N$ of probability measures $\mu_1$, \dots, $\mu_N$  so that: 
\begin{equation}\label{emergence def}
\limsup_{n\to \infty} \int \min_{1\le i\le N}  \mathsf{d}(\mathbf{e}^f_n(x), \mu_i) \, d\mu(x) \le  \epsilon \; .
\end{equation}
Let us note that when $\mu$ is $f$-invariant, then $(\mathbf{e}^f_n(x))_n$ converges to $\mathbf{e}^f(x)$ $\mu$-a.e.\ and so \eqref{emergence def} can be replaced by:
\begin{equation}
\int \min_{1\le i\le N}  \mathsf{d}(\mathbf{e}^f(x), \mu_i) \, d\mu(x) \le \epsilon\; .
\end{equation}
\end{definition}

As we will explain in \cref{s:quantization}, when the measure $\mu$ is $f$-invariant, metric emergence becomes a particular case of the classic problem of quantization (or discretization) of a measure \cite{GrafL}.

\medskip

Let us recall some examples of metric emergence.  By definition, if $(f,\mu)$ is ergodic then $\mathbf{e}^f(x)=\mu$ for  $\mu$-a.e.~$x$ and so its metric emergence $\Eme_\mu$  is identically $1$ (i.e.\ minimal).

When $X$ is a compact manifold $M$, the metric emergence will be canonically considered for $\mu = \Leb$, the Lebesgue measure of $M$ (that is, the probability measure corresponding to a fixed normalized smooth positive volume form). The map $f$ is called \emph{conservative} if it leaves the Lebesgue measure invariant. The group of conservative $C^r$-diffeomorphisms  is denoted by $\Diff_\Leb^r(M)$. 

There are well-studied subsets of $\Diff^r_\Leb(M)$ consisting of ergodic diffeomorphisms: uniformly hyperbolic dynamics, quasi-periodic mappings (e.g.\ minimal translations of tori), and many classes of partially hyperbolic dynamics \cite{BuW, ACW, Obata}.

For a while, Boltzmann's ergodic hypothesis 
prevailed and typical Hamiltonian dynamical systems were believed to be ergodic \cite{BK, Dumas}. 
However, KAM (Kolmogorov--Arnold--Moser) theory revealed that every perturbation of certain integrable systems displays infinitely many invariant tori filling a set of positive Lebesgue measure, in each of which the dynamics is an ergodic rotation. Thus the ergodic hypothesis was refuted. 
This phenomenon also showed that a typical symplectic diffeomorphism is in general not ergodic, since nearby its totally elliptic periodic points it is Hamiltonian and smoothly approximable by an integrable system.
As we will explain later (see  \cref{c:robust_poly}), the metric emergence of systems displaying KAM phenomenon is at least polynomial:
\begin{equation}\tag{$\ge P$}\label{e.liminf_poly}
\liminf_{\epsilon\to 0} \frac{\log \Eme_{\Leb}(f)(\epsilon)}{-\log \epsilon} \ge 1 \, .
\end{equation}

Another phenomenon, discovered by Newhouse \cite{Ne1}, is the co-existence of infinitely many invariant open sets, each of which having an asymptotically constant empirical function, 
so that the corresponding probability measures can approximate any invariant ergodic measure supported on a certain non-trivial hyperbolic compact set. 
This phenomenon occurs generically in many categories of dynamical systems: see \cite{Ne2, Duarte, Bu, BD, DNP, Biebler}.
The Newhouse phenomenon has been recently shown to be typical in the sense of Kolmogorov: see \cite{Be1, Berger}.

In view of \cref{emergence:topo:example}, one might believe that the metric emergence of systems displaying Newhouse phenomenon can have super-polynomial growth. 
In the paper \cite{Berger}, it is actually conjectured that super-polynomial growth is typical in open sets of many categories of dynamical systems. We prove one step toward this conjecture by showing (in \cref{section:metric:emer}) the existence of a smooth (that is, $C^\infty$) conservative flow 
with stretched exponential metric emergence:

\begin{theorem}\label{t:main}
There exists a smooth conservative flow $(\Phi^t)_t$ on the  annulus $\A \coloneqq \R/\Z \times [0,1]$
such that for every $t\neq 0$ the emergence of $f = \Phi^t$ is stretched exponential with (maximal) exponent $d=2$: 
\begin{equation}\tag{$S\exp^d$}\label{e.liminf_stretchted}
\lim_{\epsilon\to 0} \frac{\log \log \Eme_{\Leb}(f) (\epsilon)}{-\log \epsilon} = d
\; .\end{equation}
\end{theorem}

\begin{remark}\label{entropy vs emergence}
We recall that surface flows have zero topological entropy. Thus the previous theorem provides an example of smooth conservative dynamics with high emergence but zero topological entropy.
\end{remark}

\begin{question}
Is there a smooth conservative surface map $f$ such that:
$$\liminf_{\epsilon\to 0} \epsilon^2 \log \Eme_{\Leb}(f) (\epsilon)>0\, ?$$
\end{question}

Actually the proof of this theorem can be adapted to show the existence of smooth conservative flow $(\Phi^t)_t$ of a compact manifold $M$ of any dimension $d\ge 2$ such that the metric emergence of $f = \Phi_t$, $t\neq 0$ satisfies \eqref{e.liminf_stretchted}.

\medskip

We recall that a conservative map is \emph{$C^r_\Leb$-weakly stable} if every conservative mapping in a $C^r$ neighborhood has only hyperbolic periodic points (i.e.\ points $x = f^p(x)$ for which the eigenvalues of $Df^p(x)$ have moduli different than~$1$). Such mappings are conjecturally uniformly hyperbolic (and so structurally stable): see \cite{berger-turaev}, \cite[Conj.~2]{Mane82}. 
In that paper, it was shown that any conservative surface diffeomorphism which is not weakly stable can be $C^\infty$-approximated by one with positive metric entropy.
Here we obtain a stronger emergence counterpart of this result:

\begin{theorem}\label{coroC}
A $C^\infty$-generic, conservative, surface diffeomorphism $f$ 
either is weakly stable or has a metric emergence with $\limsup$ stretched exponential with exponent $d=2$: 
\begin{equation}\tag{$\overline S\exp^d$} \label{e.limsup_stretchted}
\limsup_{\epsilon\to 0} \frac{\log \log \Eme_{\Leb}(f) (\epsilon)}{-\log \epsilon}  = d \; .
\end{equation}
\end{theorem}

\begin{remark}
\Cref{coroC} certainly requires some appropriate degree of differentiability, and is completely false for the $C^0$ category -- indeed, generic volume-preserving homeomorphisms of a compact manifold are ergodic \cite{OxtobyUlam} and therefore have minimal metric emergence.
\end{remark}

With a relatively simple modification of the proof of the previous \lcnamecref{coroC}, we also obtain its dissipative (i.e.\ non conservative) counterpart:

\begin{theorem}\label{coroD}
For every $r\in [1, \infty]$ and for every surface $M$, there exists a non-empty open set $\cU\subset \Diff^r(M)$, such that a generic map $f\in \cU$ has metric emergence $\Eme_\Leb(f)$ that satisfies \eqref{e.limsup_stretchted}
with $d=2$.
\end{theorem}

These results prove a weak version of Conjecture~A of \cite{Berger} for the classes of smooth conservative and non-conservative surface diffeomorphisms. This conjecture posits the existence of many open classes of dynamics for which super-polynomial emergence is typical in many senses
 (including Kolmogorov's).
 In this regard, let us note that it is an open question whether Newhouse phenomenon implies typically high emergence.

\medskip

Our results also make it clear that emergence and entropy are completely unrelated. 
On one hand, a uniformly hyperbolic, conservative map has positive metric entropy but minimal metric emergence (identically equal to~$1$), since the volume measure is ergodic. 
Furthermore, a construction of Rees and B\'eguin--Crovisier--Le~Roux \cite{BCLR} yields a homeomorphism which is uniquely ergodic (and so has minimal topological and metric emergences) but has positive topological entropy.
On the other hand, \cref{t:main} gives an example of conservative dynamics with stretched exponential emergence but (as noted in \cref{entropy vs emergence}) with zero topological entropy 
and in particular (by the entropy variational principle) with zero metric entropy. 

\medskip

As we will show in \cref{s:quantization}, the metric emergence of any invariant measure is at most the topological emergence (see \cref{compa metric_topo emergence}).
Furthermore, we will prove that the latter upper bound is asymptotically attained, therefore obtaining the following statement that mirrors the entropy variational principle:

\begin{theorem}[Variational Principle for Emergence] \label{t:var_princ}
For every continuous self-map $f$ of a compact metric space $X$, 
there exists an invariant probability measure $\mu$ such that:
\[
\left\{
\begin{aligned}
\limsup_{\epsilon\to 0} \frac{\log \log  \Eme_\mu(f)(\epsilon)}{-\log\epsilon} &= \limsup_{\epsilon\to 0} \frac{\log \log  \Eme_\mathrm{top}(f)(\epsilon)}{-\log\epsilon}\, , \\ 
\liminf_{\epsilon\to 0} \frac{\log \log  \Eme_\mu(f)(\epsilon)}{-\log\epsilon}  &= \liminf_{\epsilon\to 0} \frac{\log \log  \Eme_\mathrm{top}(f)(\epsilon)}{-\log\epsilon} \, . \end{aligned}\right.\]
\end{theorem}

\begin{question}
Can we find an invariant measure $\mu$ such that $\Eme_{\mu}(f)\sim \Eme_\mathrm{top}(f)$ (that is, these two functions of $\epsilon$ are asymptotic as $\epsilon \to 0$)?
\end{question}

\subsection*{Organization of the paper}

In \cref{s:mo} we discuss covering numbers and the related concepts of box-counting dimension (the exponent when covering numbers obey a power law) and metric order (the exponent when covering numbers obey a stretched exponential law); furthermore, we define precisely the Wasserstein and L\'evy--Prokhorov metrics on space $\cM(X)$ of probability measures on a compact metric space $X$, and show that the metric order of $\cM(X)$ coincides with the box-counting dimension of $X$ (\cref{t:mo_box}).
Since topological emergence of a dynamical system is defined as the covering number function of its set of invariant probability measures, we obtain a simple upper bound for the growth rate of topological emergence of a dynamical system in terms of the dimension of the phase space. In \cref{s:top_em_examples} we exhibit classes of examples where this upper bound is attained: this is the content of \cref{emergence:topo:example}. The proof uses elementary properties of Gibbs measures. 

In \cref{s:quantization} we recall the notion of quantization of probability measures. We define quantization numbers, which express how efficiently a measure can be discretized, and are bounded from above by the covering numbers of the ambient space. We show that metric emergence of an invariant measure amounts to the quantization number function of its ergodic decomposition. We prove \cref{t:var_princ}, which says that one can always find an invariant probability measure with essentially maximal metric emergence; actually this is deduced from a more abstract result (\cref{t:fat_measure}) on the existence of measures with essentially maximal quantization numbers. 

In \cref{section:metric:emer} we construct an example of smooth conservative surface diffeomorphism such that Lebesgue measure has essentially maximal metric emergence; more precisely, we prove \cref{t:main}. In \cref{s:generic} we prove our results on genericity of high emergence for conservative and dissipative surface diffeomorphisms, \cref{coroC,coroD}; these proofs are relatively short because we make use of an intermediate result used to obtain \cref{t:main}, namely \cref{p:main}.
The proofs also use elementary versions of the KAM theorem and the persistence of normally contracted submanifolds.

In \cref{s:entropy} we recall a characterization of metric entropy due to Katok~\cite{Katok}, and show how it provides a characterization of metric entropy in terms of quantization numbers.
Katok's theorem is used in \cref{s:top_em_examples}; that is because the proof of \cref{emergence:topo:example} uses the measure of maximal dimension as an auxiliary device in the construction of large sufficiently separated sets of periodic orbits.

\subsection*{Notation}
We employ the usual notations: 
\begin{itemize}
\item $f \sim g$ means $f/g \to 1$;
\item $f \asymp g$ means $f = O(g)$ and $g = O(f)$.
\end{itemize}

\subsection*{Acknowledgements}
We are grateful to R\'emi Peyre for meticulous discussions on the covering numbers of spaces of measures, and for allowing us to include the resulting proof of \cref{t:mo_box} in this paper. 
We also thank Viviane Baladi, Abed Bounemoura, Sylvain Crovisier,  Bassam Fayad, Godofredo Iommi, Fran\c{c}ois Ledrappier, Enrique Pujals, and Michael Shub for insightful comments. 
Finally, we thank the referee for corrections and suggestions.

\section{Metric orders and spaces of measures}\label{s:mo}

\subsection{Dimension and metric order of a compact metric space}

Let $X$ be a totally bounded space,
and let $\epsilon>0$.
A subset $F \subset X$ is called:
\begin{itemize}
	\item \emph{$\epsilon$-dense} if $X$ is covered by the closed balls of radius $\epsilon$ and centers in~$F$;
	\item \emph{$\epsilon$-separated} if the distance between any two distinct points of $F$ is greater than~$\epsilon$.
\end{itemize}

Then we define the following numbers (which are finite by total boundedness):
\begin{itemize}
	\item the \emph{covering number} $D_X(\epsilon) = D(X,\epsilon)$ is the minimum cardinality of an $\epsilon$-dense set;
	\item the \emph{packing number} $S_X(\epsilon) = S(X,\epsilon)$ is the maximum cardinality of an $\epsilon$-separated set;
\end{itemize}
Precise computation of these numbers is seldom possible (see the classic \cite{Rogers} for problems of this nature). However we are only interested in the asymptotics of these numbers as $\epsilon$ tends to $0$, and so moderately fine estimates will suffice.

Covering and packing numbers can be compared as follows: 
\begin{equation}\label{e.comparison}
S_X( 2\epsilon) \le D_X(\epsilon) \le S_X(\epsilon) \, ;
\end{equation}
indeed the first inequality follows from the observation that a $2\epsilon$-separated set of cardinality $n$ cannot be covered by less than $n$ closed balls of radius~$\epsilon$, while the second inequality follows from the fact that every maximal $\epsilon$-separated set is $\epsilon$-dense.

The \emph{upper box-counting dimension} of $X$ is defined as:
$$
\overline{\dim}(X) \coloneqq \limsup_{\epsilon \to 0} \frac{\log D_X(\epsilon)}{- \log \epsilon}  \in [0,\infty] \, .
$$
Note that by inequalities \eqref{e.comparison}, it makes no difference if $D_X(\epsilon)$ is replaced by $S_X(\epsilon)$ in the definition above.
We define the \emph{lower box-counting dimension}  $\underline{\dim}(X)$ by taking $\liminf$ instead of $\limsup$.
If these two quantities coincide, they are called the \emph{box-counting dimension} of $X$ and denoted by $\dim X$. The term \emph{Minkowski dimension} is also used.
For an elementary introduction and more information, see \cite{Falconer}.

The dimensions defined above are infinite when the numbers $D_X(\epsilon)$ and $S_X(\epsilon)$ are super-polynomial with respect to $\epsilon^{-1}$. 
However, these functions are often comparable to stretched exponentials; indeed many examples of functional spaces with this property are studied in the classic work by Kolmogorov and Tihomirov \cite{KT}\footnote{See also \cite[\S~8.2.6]{heritage} for historical context.}. 
The corresponding exponent
$$
\mo(X)\coloneqq \lim_{\epsilon \to 0} \frac{\log \log D_X(\epsilon)}{- \log \epsilon} = \lim_{\epsilon \to 0} \frac{\log \log S_X(\epsilon)}{- \log \epsilon} \, ,
$$
if it exists, 
is called the \emph{metric order} of $X$, following \cite[p.~298]{KT}.
In general we define \emph{lower and upper metric orders} $\underline{\mo}(X) \le \overline{\mo}(X)$ by taking $\liminf$ and $\limsup$.

\begin{remark}\label{r.Kloeckner}
A concept similar to metric order, called \emph{critical parameter for the power-exponential scale}, was introduced and studied by Kloeckner \cite{Klo12, Klo15}. Its definition is more akin to the Hausdorff dimension. 
\end{remark}

\begin{remark}\label{rel_cov}
If $Y$ is a subset of $X$ then define the \emph{relative covering number} $D_X(Y,\epsilon)$ as the minimal number of closed $\epsilon$-balls in $X$ whose union covers~$Y$. Note that:
\[D_X(Y,\epsilon)\le D(Y,\epsilon)\le D_X(Y,\epsilon/2) \, .\]  
Therefore dimension and metric order of subsets of $X$ can be also computed using relative covering numbers.
\end{remark}

\subsection{Spaces of measures}\label{ss:distproba}

Let $(X,\mathsf{d})$ be a compact metric space.
Let $\cM(X)$ be the space of Borel probability measures on $X$, endowed with the weak topology and therefore compact.
There are many different ways of metrizing the weak topology.
We will consider two types of metrics in $\cM(X)$: the Wasserstein distances and the L\'evy--Prokhorov distance (defined below).
These metrics respect the original metric on $X$, in the sense that the map $x \mapsto \delta_x$ (where $\delta_x$ is the Dirac probability measure concentrated at the point $x$) becomes an \emph{isometric} embedding of $X$ into $\cM(X)$.

Given two measures $\mu$, $\nu \in \cM(X)$, a \emph{transport plan} (or \emph{coupling}) from $\mu$ to $\nu$ is a probability measure $\pi$ on the product $X \times X$ such that $(p_1)_* \pi = \mu$ and $(p_2)_* \pi = \nu$, where $p_1$, $p_2 \colon X \times X \to X$ are the canonical projections. (We say that $\mu$ and $\nu$ are the \emph{marginals} of $\pi$.)
Such transport plans form a closed and therefore compact subset $\Pi(\mu,\nu)$ of $\cM(X \times X)$.
For any real number $p \ge 1$, the \emph{$p$-Wasserstein distance} between $\mu$ and $\nu$ is defined as:
$$
\mathsf{W}_p(\mu,\nu) \coloneqq \inf_{\pi \in \Pi(\mu,\nu)} \left( \int_{X \times X} [\mathsf{d}(x,y)]^p d \pi(x,y) \right)^{1/p} \, ,
$$
(The integral in this formula is called the \emph{cost} of the transport plan $\pi$ with respect to the \emph{cost function} $\mathsf{d}^p$. The infimum is always attained, i.e., an \emph{optimal} transport plan always exists.) It can be shown that $\mathsf{W}_p$ is a metric on $\cM(X)$ which induces the weak topology: see e.g.\ \cite[Theorems 7.3 and 7.12]{Villani}.

The \emph{L\'evy--Prokhorov distance} between two measures $\mu$, $\nu \in \cM(X)$ is denoted $\mathsf{LP}(\mu,\nu)$ and is defined as the infimum of $\epsilon>0$ such that for every Borel set $E \subset X$, if $V_\epsilon(E)$  denotes the $\epsilon$-neighborhood of $E$, then:
$$
\nu(E) \le \mu(V_\epsilon(E)) + \epsilon \qand
\mu(E) \le \nu(V_\epsilon(E)) + \epsilon \, .
$$
For a proof that $\mathsf{LP}$ is a metric on $\cM(X)$ and that induces the weak topology, see \cite[p.~72]{Bill}.

The L\'evy--Prokhorov distance can also be characterized in terms of transport plans:
it equals the infimum of $\epsilon>0$ such that for some $\pi \in \Pi(\mu,\nu)$, the set $\{(x,y) \in X \times X \st \mathsf{d}(x,y) > \epsilon \}$ has $\pi$-measure less than $\epsilon$; this is Strassen's theorem: see \cite[p.~74]{Bill} or \cite[p.~44]{Villani}.

The family of Wasserstein metrics are not Lipschitz-equivalent to one another nor to the L\'evy--Prokhorov metric. On the other hand, the following H\"older comparisons hold:
\begin{alignat}{3}
\mathsf{W}_q &\le \mathsf{W}_p &&\le (\diam X)^{1-\frac{q}{p}} \mathsf{W}_q^{\frac{q}{p}}  &\quad &\text{if $1\le q \le p$;} 
\label{Wp_monotone}
\\ 
\mathsf{LP}^{1+\frac{1}{p}} &\le \mathsf{W}_p &&\le (1+(\diam X)^p)^{\frac{1}{p}} \mathsf{LP}^{\frac{1}{p}} \, ; \label{compaWpLP}
\end{alignat}
see \cite[p.~210]{Villani}, \cite[Theorem~2]{GiS}.

\subsection{Metric order of spaces of measures} \label{section:mo:was}

The following result relates the lower and upper metric orders of Wasserstein space with the lower and upper box-counting dimensions of the underlying metric space:
\begin{otherthm} \label{t:mo_box}  
For any compact metric space $X$ and any $p \ge 1$, we have:
$$
\underline{\dim}(X) \le
\underline{\mo} (\cM(X), \mathsf{W}_p) \le \overline{\mo} (\cM(X), \mathsf{W}_p) \le \overline{\dim}(X) \, . 
$$
In particular the metric order $\mo (\cM(X), \mathsf{W}_p)$ exists and equals the box-counting dimension $\dim X$ whenever the latter exists.
\end{otherthm}


Actually, the rightmost inequality in the \lcnamecref{t:mo_box} is a consequence of a more precise result of Bolley--Guillin--Villani \cite{BGV} (details will be provided below), while a variation of the leftmost inequality was obtained by Kloeckner~\cite[Theorem~1.3]{Klo15}. Here we will present a proof of the leftmost inequality which was obtained jointly with R\'emi Peyre.

\begin{remark}\label{boundM(X)LP} 
The exact same statement also holds for the  L\'evy--Prokhorov metric, as a consequence of \cite[Lemmas 1 and A1]{KZ}.
\end{remark}

\begin{remark}\label{r.KT}
Other examples where the metric order of a functional space equals the dimension of the underlying space can be found in \cite{KT}, namely uniformly bounded uniformly Lipschitz functions on an interval \cite[p.~288]{KT}, or on more general sets \cite[p.~307]{KT}.
\end{remark}

\begin{proof}[Proof of the rightmost inequality in \cref{t:mo_box}]
By \cite[Theorem~A.1]{BGV}\footnote{See \cite[Lemma~4]{Nguyen} for a related result.}, there exists $C>0$ such that:
$$
D\big((\cM(X),\mathsf{W}_p\big), \epsilon) \le (C/\epsilon)^{p D_X(\epsilon/2)} \, .
$$
So:
\begin{multline*}
\frac{\log \log D\big((\cM(X),\mathsf{W}_p),\epsilon\big)}{\log \epsilon^{-1}} \le
\frac{\log(\log C + \log \epsilon^{-1})}{\log \epsilon^{-1}}  \\ +
\frac{\log 2 + \log \epsilon^{-1}}{\log \epsilon^{-1}} \, \frac{\log D_X(\epsilon/2) + \log p}{\log(2/\epsilon)}
\, .
\end{multline*}
Taking $\limsup$ as $\epsilon \to 0$ we obtain $\overline{\mo} (\cM(X), \mathsf{W}_p) \le \overline{\dim}(X)$.
\end{proof}

The remaining part of \cref{t:mo_box} will be obtained as a consequence of a more general result that allows us to estimate the lower metric order of other spaces of measures. 

Let us say that two probability measures $\mu$, $\nu$ on $X$ are \emph{$\epsilon$-apart} if  their supports are $\epsilon$-apart in the following sense:
\[
\min\{\mathsf{d}(x,y) \mid {x\in \supp \mu , \ y\in \supp \nu}\}\ge \epsilon\; .
\] 

\begin{otherthm} \label{t:lower_general}
Let $X$ be a compact metric space.
Let $\cC$ be a convex subset of $\cM(X)$.
For each $\epsilon>0$, let $A(\cC,\epsilon)$ denote the maximal number of pairwise $\epsilon$-apart measures in $\cC$.
Then, for any $p\ge 1$,
$$
\underline{\mo}(\cC, \mathsf{W}_p) \ge \liminf_{\epsilon \to 0} \frac{\log A(\cC, \epsilon)}{-\log \epsilon} \, .
$$
The same inequality holds for the distance $\mathsf{LP}$.
\end{otherthm}

\begin{proof}[Proof of the leftmost inequality in \cref{t:mo_box}]
We apply \cref{t:lower_general} with $\cC = \cM(X)$.
If $\{x_1, \dots, x_N\}$ is an $\epsilon$-separated subset of $X$ then the Dirac measures $\delta_{x_1}$, \dots, $\delta_{x_N}$ are pairwise $\epsilon$-apart. 
This observation shows that $A(\cM(X), \epsilon) \ge S(X,\epsilon)$.
The result follows. 
\end{proof}

To prove \cref{t:lower_general}, we will need the following elementary large-deviations estimate (see e.g.\ \cite[p.~32]{GrimS} for a proof): 

\begin{lemma}[Bernstein inequality]\label{l:Bernstein}
Let $H_n$ (a random variable) be the number of heads on $n$ tosses of a fair coin.
Then for any $\delta>0$,
$$
\mathrm{Prob} \left[ \frac{H_n}{n} \le \frac{1}{2} - \delta \right] \le e^{-\frac{\pi}{4} \delta^2 n} \, .
$$
\end{lemma}

\begin{proof}[Proof of \cref{t:lower_general} (with R\'emi Peyre)]  
Fix $\epsilon>0$, and let $N \coloneqq 8 \lfloor A(\cC,\epsilon)/8 \rfloor$. Observe that $A(\cC,\epsilon)-7 \le N \le A(\cC,\epsilon)$, and so we can find measures $\nu_1, \dots, \nu_N \in \cC$ that are pairwise $\epsilon$-separated.

Denote
\begin{equation}\label{e.F}
F \coloneqq \Bigl\{f \colon \{1,\dots,N\} \to \{0, 1\}\ \Big|\ \sum_{i=1}^N f (y) = \frac{N}{2}\Bigr\}.
\end{equation}
We endow $F$ with the Hamming distance:
\[
\mathsf{Hamm} (f, g) \coloneqq \card \{i \in  \{1,\dots,N\} |\ f (i) \neq g (i)\}
\]
(which is always an even number between $0$ and $N$).
Let us estimate the cardinality of a ball $B$ of radius $N/4$ in $F$ and centered at some $f$.
If $g$ is an element of $B$, that is, $k \coloneqq \frac{1}{2} \mathsf{Hamm} (f, g) \le N/8$, then there are exactly $k$ elements of $f^{-1}(\{0\})$ and $k$ elements of $f^{-1}(\{1\})$ at which $g$ differs from $f$.
As both sets $f^{-1}(\{0\})$ and $f^{-1}(\{1\})$ have cardinality $N/ 2$, we obtain:
\[
\card B =
\sum_{k = 0}^{N/8} \binom{N/2}{k}^2 \leq \biggl[ \sum_{k = 0}^{N/8} \binom{N/2}{k}\biggr]^2 
\, .
\]
The quantity between square brackets equals $2^{N/2}$ times the probability of obtaining at most $N/8$ heads on $N/2$ tosses of a fair coin. By \cref{l:Bernstein}, this probability is at most $e^{-\frac{\pi}{4} \left(\frac{1}{4}\right)^2 \frac{N}{2}}$.
So 
\begin{equation}\label{e.ball}
\card B \le 2^N e^{- \frac{\pi}4 \cdot\frac{N}{4^2}}  \, . 
\end{equation}

Choose a maximal $N/4$-separated subset $F'$ of $F$.
Then $F'$ is $N/4$-dense, that is, the balls of radius $N/4$ with centers in $F'$ form a covering of $F$.
The cardinality of $F$ itself is $\binom{N}{N/2} \ge (2N)^{-1/2} 2^N$ (by Stirling's formula). 
Using \eqref{e.ball}, we conclude that
\begin{equation}\label{e.Fprime}
\card F' \geq \frac{\card{F}}{\card B} \geq (2N)^{-1/2} e^{\pi N/4^3} \, .
\end{equation}

Now, for each $f \in F'$, consider the measure:
\[
\mu_f \coloneqq \frac{2}{N} \sum_{i=1}^N f(i) \nu_i \, ,
\]
which by convexity belongs to $\cC$. 
Consider the subset $\cF \coloneqq \{ \mu_f \mid f \in F'\}$ of $\cC$, which has the same cardinality as~$F'$. 
This set has the following property, whose proof will be given later:

\begin{claim}\label{sep_estimate}
The set $\cF$ is $4^{-1 / p} \epsilon$-separated with respect to the Wasserstein distance~$\mathsf{W}_p$.
\end{claim}

In particular, $S((\cC, \mathsf{W}_p), 4^{-1 / p} \epsilon) \geq \card F'$.
On the other hand, it follows from \eqref{e.Fprime} that $\card F' \ge e^{cN}$ for all sufficiently large $N$, where $c>0$ is a constant.
So:
$$
\frac{\log \log S ((\cC, \mathsf{W}_p), 4^{-1/p}\epsilon)}{- \log(4^{-1/p}\epsilon)}\ge
\frac{\log N + \log c}{- \log \epsilon +\frac1p \log 4} 
$$
Since $N\ge A(\cC,\epsilon)-7$, taking $\liminf$ as $\epsilon \to 0$ we obtain the conclusion of the \lcnamecref{t:mo_box} for the 
Wasserstein distance~$\mathsf{W}_p$.

As regards the L\'evy--Prokhorov distance $\mathsf{LP}$, 
inequalities \eqref{compaWpLP} allows us to compare it with the $\mathsf{W}_1$ distance, and so \cref{sep_estimate} implies that $\cF$ is $(4(1+\diam X))^{-1} \epsilon$-separated with respect to $\mathsf{LP}$, which allows us to conclude as before.

This completes the proof of \cref{t:mo_box}, modulo the \lcnamecref{sep_estimate}.
\end{proof}

\begin{proof}[Proof of \cref{sep_estimate}]
Fix two distinct elements $f$, $g$ of $F'$, and let us estimate $\mathsf{W}_p(\mu_f,\mu_g)$.
Let $S_f$ and $S_g$ be the supports of $\mu_f$ and $\mu_g$, respectively.

We claim that:
$$
(x,y) \in (S_f \setminus S_g) \times S_g \  \Rightarrow \ \mathsf{d}(x,y) \ge \epsilon \, .
$$
Indeed, if $y \in S_g$ then $y \in \supp \nu_j$ 
for some $j \in \{1,\dots, N\}$ such that $g(j)=1$,
while if $x \in S_f \setminus S_g$ then $x \in \supp \nu_i$ 
for some $i \in \{1,\dots, N\}$ such that $f(i)=1$ and $g(i)=0$; in particular, $i\neq j$.
So $\nu_i$ and $\nu_j$ are $\epsilon$-apart, which guarantees that $\mathsf{d}(x,y) \ge \epsilon $, as claimed.

Also note that:
$$
\mu_f(S_f \setminus S_g) = \frac{2}{N} \card \big\{i \in \{1,\dots, N\} \mid f(i)=1, g(i)=0 \big\} = \frac{\mathsf{Hamm} (f, g)}{N}  \ge \frac{1}{4} \, ,
$$
since $F$ is $N/4$-separated.

For any transport plan $\pi$ from $\mu_f$ to $\mu_g$, using the remarks above we can estimate:
\begin{multline*}
\int_{X \times X} [\mathsf{d}(x,y)]^p d\pi (x,y) 
=   \int_{S_f \times S_g}  (\cdots) 
\ge \int_{(S_f \setminus S_g) \times S_g} (\cdots) 
\\
\ge \epsilon^p \pi\big((S_f \setminus S_g) \times S_g\big) 
=   \epsilon^p \mu_f(S_f \setminus S_g) 
\ge \frac{\epsilon^p}{4} \, .	
\end{multline*}
So, by definition of the Wasserstein distance, we have $\mathsf{W}_p (\mu_f, \mu_g) \ge \frac{\epsilon}{4^{1/p}}$, completing the proof of the \lcnamecref{sep_estimate}.
\end{proof}

\section{Examples of dynamics with high topological emergence} \label{s:top_em_examples}

Let $f$ be a continuous self-map of a compact metric space $X$. 
We recall that $\cM^\mathrm{erg}_f(X)$ denotes the space of invariant ergodic probability measures. 

As explained in the introduction, the \emph{topological emergence} of $f$  is the relative covering number of $\cM^\mathrm{erg}_f(X)$ (defined in \cref{rel_cov})
endowed either with a Wasserstein distance~$\mathsf{W}_p$, $1\le p<\infty$, or the L\'evy-Prokhorov distance~$\mathsf{LP}$, that is:
\begin{equation}\label{e.def_top_emergence}
\Eme_\mathrm{top}(f)(\epsilon) \coloneqq D_{\cM(X)}(\cM^\mathrm{erg}_f, \epsilon) \, .
\end{equation}
We are concerned with the asymptotic behavior of this function for small $\epsilon$.
Since $\cM^\mathrm{erg}_f(X)$ is included in $\cM(X)$, by \cref{t:mo_box} and \cref{boundM(X)LP} we have:
\begin{equation}\label{e.upper_bound_top_emergence}
\limsup_{\epsilon\to 0} \frac{\log \log \Eme_\mathrm{top}(f)(\epsilon)}{-\log \epsilon} = \overline{\mo}(\cM^\mathrm{erg}_f(X)) \le 
\overline{\mo}(\cM(X))\le \overline \dim (X).
\end{equation}
Sometimes, this bound is far from being optimal. For instance, when $f$ is uniquely ergodic, then 
$\Eme_\mathrm{top}(f)(\epsilon)=1$ does not grow at all.  
If $f$ is the identity of $X$, then  $\cM^\mathrm{erg}_f(X)$ is isometric to $X$, and so $\Eme_\mathrm{top}(f)(\epsilon)$ is comparable to $\epsilon^{-d}$ if $X$ has well defined box-counting dimension $d$. 

On the other hand, \cref{emergence:topo:example} gives examples of hyperbolic compact sets for which the above bound is optimal. Let us explain and prove them.

\subsection{Conformal expanding repellers}
Let $M$ be a Riemannian manifold, $U$ an open subset of $M$ and  $f \colon U \to M$ be $C^{1+\alpha}$ map which leaves invariant a compact subset $K$ of $U$ (i.e. $f^{-1}(K)=K$). 
We say that $(K,f)$ is a \emph{conformal expanding repeller} if  $f$ is conformal and expanding at $K$: for each $x\in K$, the derivative $Df(x)$ expands the Riemannian metric by a scalar factor greater than~$1$.
Then its box-counting dimension $\dim(K)$ is well-defined, and it equals the Hausdorff dimension: see \cite[Corol.~9.1.7]{PU}.

\begin{otherthm}\label{t:example_conformal}
Let $(K,f)$ be a conformal expanding repeller of dimension~$d$. 
Then the topological emergence of $f|K$ is stretched exponential with exponent~$d$:
\[
\lim_{\epsilon\to 0} \, \frac{\log \log \Eme_\mathrm{top}(f|K) (\epsilon)}{-\log \epsilon} = d\; .
\]
\end{otherthm}

\begin{proof}
First, we can assume that $K$ is transitive since it is always a finite disjoint union of transitive sets; moreover, up to taking an iterate of $f$, we can suppose that $f|K$ is topologically mixing -- see \cite[Thm.~3.3.8]{PU}.

By standard results \cite[{\S}9.1]{PU}, there exists an invariant ergodic probability measure $\mu$ supported on $K$ of \emph{maximal dimension}. The Lyapunov exponent $\chi_\mu\coloneqq  \int \log \| Df \| \, d\mu $ 
and metric entropy $h_\mu$ are related as follows:
\begin{equation}\label{entropy:formula}
	\chi_\mu \cdot d = h_\mu \, , \quad \text{where } d = \dim K \,.
\end{equation}

Let $\rho_0 > 0$ be such that $U$ contains the $\rho_0$-neighborhood of $K$. Reducing $\rho_0$ if necessary, 
there exists  $\lambda>1$  such that $f$ is $\lambda$-expanding on the $\rho_0$-neighborhood of $K$, in the sense that $ \|Df^{-1}\|^{-1} \ge \lambda$. Then 
we have the following property \cite[{\S}4.1]{PU}: for all $x \in K$ and all $n\ge 1$, the connected component { $V^n_x$} of $x$ in the preimage by $f^n$ of the (Riemannian) ball $B(f^n(x),\rho_0)$ is included in $B(x,\lambda^{-n} \rho_0)$. Moreover $V^n_x$ is sent by $f^n$ diffeomorphically onto $B(f^n(x),\rho_0)$. Note that $\rho_0$ is an expansiveness constant for $f|K$, in the sense that if $x \neq y$ then there exists $n\ge 0$ such that $\mathsf{d}(f^n(x),f^n(y)) \ge \rho_0$.

Let $\mathsf{d}$ be the metric on $M$ induced by the Riemannian structure, and for each $n \ge 1$, let $\mathsf{d}_n$ denote the time-$n$ Bowen metric on $K$, defined by:
\[
\mathsf{d}_n(x,y) \coloneqq \max_{0 \le i < n} \mathsf{d}(f^i(x),f^i(y)) \, .
\] 
By the bounded distortion property \cite[Lemma~4.4.2]{PU}, 
there exists a constant  $C_0>1$ such that for any $n \ge 0$, if a pair of points $(x,y) \in K\times U$ satisfies $\mathsf{d}_n(x,y) < \rho_0$ then $\|Df^{n}(y)\| \le C_0 \|Df^{n}(x)\|$.

{ Reducing $\rho_0$ if necessary, we assume that every pair of points $(x,y) \in K\times U$ such that $\mathsf{d}(x,y)<\rho_0$ can be joined by a unique geodesic segment of minimal length, denoted $[x,y]$.}

\begin{claim} \label{macroscopic_exp}
If $n \ge 1$ and $(x,y) \in K\times U$ are such that $\mathsf{d}_{n+1}(x,y) <  \rho_1 { \coloneqq C_0^{-2} \rho_0}$ then:
\begin{equation*} 
\frac{\mathsf{d}(f^{n}(x),f^{n}(y))}{\mathsf{d}(x,y)} \le 
{ C_0} \|Df^{n}(x)\| \, .
\end{equation*}
\end{claim}

\begin{proof}
Fix $x \in K$ and $n \ge 1$.
As explained above, $f^n$ maps $V^n_x$ diffeomorphically onto $B(f^n(x),\rho_0)$; let $f_x^{-n} \coloneqq (f^n|V^n_x)^{-1}$ be its inverse.
Note that $V^n_x$ is exactly the $\mathsf{d}_{n+1}$-ball of center $x$ and radius $\rho_0$.
Now consider $y \in U$ such that $\mathsf{d}_{n+1}(x,y) <  \rho_1 \coloneqq C_0^{-2} \rho_0$.
We have $\mathsf{d}(f^n(x),f^n(y)) < \rho_1$, by definition of the Bowen metric. 
Consider the geodesic segment $S \coloneqq [f^n(x), f^n(y)]$. 
Since $S$ is contained in $B(f^n(x),\rho_1) \subset B(f^n(x),\rho_0)$, the curve $f_x^{-n}(S)$ is well-defined and is contained in $V^n_x$.
Since this curve joins $x$ and $y$, we have:
\begin{align}
\mathsf{d}(x,y) \le  \mathrm{len}(f_x^{-n}(S)) 
&\le C_0 \|Df^n(x)\|^{-1} \mathrm{len}(S)  \label{e.um} \\
&< C_0 \|Df^n(x)\|^{-1} \rho_1 \notag \\
&\le C_0^{-1} \|Df^n(x)\|^{-1} \rho_0 \label{e.dois} \, ,
\end{align}
where the estimate \eqref{e.um} follows from the bounded distortion property and conformality of the derivatives, and \eqref{e.dois} follows from the definition of $\rho_1$.

We claim that the geodesic segment $[x,y]$ is contained in the interior of~$V^n_x$.
Indeed, if that is not the case, there exists 
a subsegment $[x,z]\subset  V^n_x$ such that $z \in \partial V^n_x$.
On one hand, $f^n(z) \in f^n(\partial V^n_x) \subset \partial B(f^n(x),\rho_0)$; on the other hand, using bounded distortion again,
\begin{align}
\mathsf{d}(f^n(x),f^n(z)) 
\le  \mathrm{len}(f^n([x,z]))
&\le C_0 \|Df^n(x)\| \mathsf{d}(x,z) \label{e.bootstrap}\\
&\le C_0 \|Df^n(x)\| \mathsf{d}(x,y) \notag \\ 
&< \rho_0 \qquad\text{(by \eqref{e.dois}),} \notag
\end{align}
a contradiction. 
This confirms that $[x,y]$ is contained in the interior of $V^n_x$.

We are now allowed to apply estimate \eqref{e.bootstrap} with $z=y$ and therefore conclude the validity of \cref{macroscopic_exp}.
\end{proof}

Fix a small $\delta > 0$.
By Katok's \cref{t:Katok} (see the appendix), there exists a positive number $\rho < \rho_1$ such that for all sufficiently large $n$, 
the least number $N_\mu(n,\rho,1/2)$ of balls of radii $\rho$ in the $\mathsf{d}_n$ metric necessary to cover a set of $\mu$-measure $\ge 1/2$ satisfies:
$$
N_\mu(n,\rho,1/2) > e^{(h_\mu-\delta)n} \, .
$$
For each $n\ge 1$, let $B_n$ be the set of points $x \in K$ such that 
$\| Df^n(x) \| \le e^{(\chi_\mu + \delta) n}$.
By Birkhoff theorem, if $n$ is large enough then $\mu(B_n)>1/2$.
Take a $(\mathsf{d}_n,\rho)$-separated set $F_n \subset B_n$ of maximal cardinality.
Then the balls of radii $\rho$ and centered at points in $F_n$ cover $B_n$. 
Therefore:
\begin{equation}\label{e.Fn_entropy}
\card F_n \ge N_\mu (n,\rho,1/2) > e^{(h_\mu-\delta)n} \, ,
\end{equation}
provided $n$ is large enough.

By the specification property of topologically mixing repellers (see e.g.\ \cite[Prop.~11.3.1]{VO}), 
there exists an integer $n_0 \ge 0$ (depending on $\rho$) such that for every $n$, each point $x \in F_n$ is shadowed by an $(n+n_0)$-periodic point $y\in K$ in such a way that
$\mathsf{d}_n(x,y) < \rho/2$.
Let $G_n$ be the set of periodic points $y$ obtained in this way. 
Note that $G_n$ has the same cardinality as $F_n$.
Also note that, by bounded distortion, $\| Df^n(y) \| \le C_0 \| Df^n(x) \| \le C_0 e^{(\chi_\mu + \delta) n}$ and so, if $n$ is large enough, 
\begin{equation}\label{e.Lyapunov_per}
\| Df^{n+n_0}(y) \| \le e^{(\chi_\mu + 2\delta)n} \, .
\end{equation}

Let $\Pi_n \coloneqq \bigcup_{k\ge 0} f^k(G_n)$ be the union of the orbits of the points in $G_n$.
By periodicity, the points $y \in \Pi_n$ satisfy the same estimate \eqref{e.Lyapunov_per}.

\begin{claim}
The set $\Pi_n$ is $(\mathsf{d},\epsilon_n)$-separated with $\epsilon_n \coloneqq e^{-(\chi_\mu + 3\delta)(n+1) }$, provided $n$ is large enough.
\end{claim}
 
\begin{proof}
Take a pair of distinct points $y$, $z \in \Pi_n$, and let us prove that $\mathsf{d}(y,z) > \epsilon_n$.
Both points are fixed by $f^{n+n_0}$, so, by expansiveness, there exists $k$ in the range $0 \le k < n + n_0$ such that $\mathsf{d}(f^k(y),f^k(z)) \ge \rho_1$ since $\rho_1\le \rho_0$.
Assume that $k$ is minimal.
If $k=0$ then the desired estimate is trivial, so consider $k>0$. 
Then $\mathsf{d}_k(y,z) < \rho_1$ and the following estimates hold:
\begin{alignat*}{2}
\mathsf{d}(y,z) &\ge C_0^{-1} \mathsf{d}(f^{k-1}(y),f^{k-1}(z)) \|Df^{k-1}(y)\|^{-1} &\quad&\text{(by \cref{macroscopic_exp})}\\
                &\ge C_1^{-1} \rho_1  \|Df^{k-1}(y)\|^{-1}&\quad&\text{with } C_1\coloneqq C_0 \cdot \|Df\|\\ 
				&\ge C_1^{-1} \rho_1  \|Df^{n+n_0}(y)\|^{-1}  &\quad&\text{(since $f$ is expanding)}\\
				&>   C_1^{-1} \rho_1 \cdot e^{-(\chi_\mu + 2\delta) n} &\quad&\text{(by \eqref{e.Lyapunov_per}).}
\end{alignat*}
This implies the sough inequality when $n$ is large enough.
\end{proof}

So any two distinct ergodic measures supported in the finite invariant set $\Pi_n$ are $\epsilon_n$-apart (in the sense defined in \cref{section:mo:was}). The number $A_n$ of such ergodic measures satisfies:
$$
A_n \ge \frac{\card G_n}{n + n_0} = \frac{\card F_n}{n + n_0} \ge  e^{(h_\mu-2\delta) n}
$$
if $n$ is sufficiently large (by \eqref{e.Fn_entropy}).

Now, given $\epsilon>0$ sufficiently small, take $n$ such that $\epsilon_n \le \epsilon < \epsilon_{n-1}$.
Consider the convex set $\cC \coloneqq \cM_f(K)$ of all $f$-invariant measures;
then, in the notation of \cref{t:lower_general}, we have 
$A(\cC,\epsilon) \ge A(\cC,\epsilon_n) \ge A_n$ and so
$$
\frac{\log A(\cC,\epsilon)}{- \log \epsilon} \ge \frac{\log A_n}{- \log \epsilon_{n-1}} \ge \frac{h_\mu - 2\delta}{\chi_\mu + 3\delta} \, . 
$$
So \cref{t:lower_general} yields $\underline{\mo}(\cM_f(K)) \ge (h_\mu - 2\delta)/(\chi_\mu + 3\delta)$.
As $\delta$ is arbitrarily close to $0$, we conclude that $\underline{\mo}(\cM_f(K))$ is at least $h_\mu/\chi_\mu$, which by \eqref{entropy:formula} equals $d = \dim K$.

As a consequence of specification (see \cite[Thrm.~11.3.4]{VO}), the closure of $\cM^\mathrm{erg}_f(K)$ equals $\cM_f(K)$. Therefore: 
$$
\underline{\mo}(\cM_f^\mathrm{erg}(K)) = \underline{\mo}(\cM_f(K)) \ge d \, .
$$
On the other hand, $\overline{\mo}(\cM_f^\mathrm{erg}(K)) \le d$ by \eqref{e.upper_bound_top_emergence}.
So $\mo(\cM_f^\mathrm{erg}(K)) = d$, as we wanted to show.
\end{proof}


\subsection{Hyperbolic sets of conservative surface diffeomorphisms}
Let $M$ be a surface and let $f \colon M \to M$ be a $C^{1+\alpha}$ diffeomorphism. Let $K \subset M$ be a \emph{hyperbolic} set for $f$. This means that $K$ is an invariant compact set $K$ and there exists an invariant splitting $E^\mathrm{s}\oplus E^\mathrm{u}$ 
of the tangent bundle $TM$ of $M$ restricted to $K$ such that the line bundles $E^\mathrm{s}$ and $E^\mathrm{u}$ are respectively contracted and expanded. In other words, there exists $\lambda>1$ such that for every $z\in K$: 
\[\left\{\begin{array} {ccc} D_zf(E^\mathrm{s}_z)=E^\mathrm{s}_{f(z)}&\& & \|D_zf|E^\mathrm{s}\|^{-1}>\lambda \, ,\\
Df_z(E^\mathrm{u})=E^\mathrm{u}_{f(z)}&\& &  \|D_zf|E^\mathrm{u}\|>\lambda\; .\end{array}\right.\]
Let us assume moreover that the compact set $K$ is \emph{locally maximal}, that is, it admits a neighborhood $U$ such that $K=\bigcap_{n\in \Z} f^n(U)$.

\begin{otherthm}\label{t:example_surface}
If $f$ is conservative then the topological emergence of $f|K$ is stretched exponential with exponent $d \coloneqq \dim(K)$:
\[
\lim_{\epsilon\to 0} \, \frac{\log \log \Eme_\mathrm{top}(f|K) (\epsilon)}{-\log \epsilon} = d\; .
\]
\end{otherthm}

\begin{proof}
First, we can assume that $K$ is transitive since it is always a finite disjoint union of such sets; moreover, up to taking an iterate of $f$, we can consider that $f|K$ is topologically mixing -- see \cite[Thm.~18.3.1, p.~574]{KH}.

From standard results on dimension theory of hyperbolic sets (see e.g.\ \cite[Thrm.~22.2]{Pesin}),
the box-counting dimension $d \coloneqq \dim K$ is well defined, and it equals $d^\mathrm{s} +d^\mathrm{u}$, where $d^\mathrm{s}$ (resp.\ $d^\mathrm{u}$) is the box-counting dimension of $K$ intersected with any local stable (resp.\ unstable) manifold. Moreover, for every $\star\in \{\mathrm{u},\mathrm s\}$, there exists an invariant ergodic probability measure $\mu^\star $ supported on $K$ of \emph{maximal $\star$-dimension}. The  Lyapunov exponent $\chi_{\mu^\star}\coloneqq  \int \log \| Df |E^\star\| \, d\mu $  and the metric entropy $h_{\mu^\star}$ are related as follows:
\begin{equation}\label{entropy:formula2}
	\chi_{\mu^\star} \cdot d^\star  = h_{\mu^\star}\,.
\end{equation}
Those measures are obtained as the unique equilibrium states for the functions:
$$
\phi^\mathrm{s}(x) \coloneqq - \log \| Df(x) | E^\mathrm{s}_x \|  \, , \qquad  
\phi^\mathrm{u}(x) \coloneqq \log \| Df(x) | E^\mathrm{u}_x \|  \, .
$$
The dynamics being consevative, the functions $\phi^\mathrm{s}$ and $\phi^\mathrm{u}$ are cohomologous. Thus by uniqueness of equilibria:
\[\mu^\mathrm s=\mu^\mathrm u=:\mu \qand  -\chi_{\mu^\mathrm s}=\chi_{\mu^\mathrm u}=: \chi_\mu\, ,\]
and so by \eqref{entropy:formula2} and using $d=d^\mathrm u+d^\mathrm s$:
\begin{equation}\label{entrop_formula_conserv} \chi_\mu \cdot \frac d2 =h_\mu\, .\end{equation}
 
Let us fix continuous families of local stable and unstable manifolds $(W^\mathrm{s}_\mathrm{loc} (x))_{x\in K}$ and $(W^\mathrm{u}_\mathrm{loc} (x))_{x\in K}$, small enough to be $\lambda^{-1}$-contracted by respectively $f$ and $f^{-1}$. Furthermore, whenever $x$ and $y\in K$ are close enough, then 
 $W^\mathrm{u}_\mathrm{loc}(x)$ intersects $W^\mathrm{s}_\mathrm{loc}(y)$ at a unique point, called the \emph{bracket} of $x$ and $y$ and denoted $[x,y]$. 
By local maximality of $K$, the point $[x,y]$ belongs to $K$. 

Let $\tilde{\mathsf{d}}_n$ denote the bilateral Bowen metric on $M$, defined by:
\[
\tilde{\mathsf{d}}_n(x,y) \coloneqq \max_{-n < i < n} \mathsf{d}(f^i(x),f^i(y)) \, .
\] 
We denote by $\mathsf d^\mathrm u$ (resp.\ $\mathsf d^\mathrm s$) the distance along the local unstable (resp.\ stable) manifolds. Using the contraction along the local stable and unstable manifolds by $f$ and $f^{-1}$, we obtain:
\begin{claim}\label{expansivityus} There exists $\rho_0>0$ small and $c>0$ such that for any  $x\neq y\in K$ which are $\rho_0$-close, there exists $k\ge 1$ such that $\tilde{\mathsf{d}}_{k}(x,y)<\rho_0 \le \tilde{\mathsf{d}}_{k+1}(x,y)$ and:
\[\mathsf d^\mathrm u(f^k(x),  f^k([x, y]))> c \cdot \rho_0\quad \text{or} \quad
\mathsf d^\mathrm s(f^{-k}(x), f^{-k}([y, x]))> c \cdot  \rho_0\; .\]
 \end{claim}

\medskip 

By the bounded distortion property \cite[Prop.~22.1]{Pesin},
there exists a constant  $C_0>1$ such that for any $n \ge 0$ and $x\in K$, the following estimates hold for every $y\in M$ such that $\tilde{\mathsf{d}}_n(x,y)<\rho_0$:
\begin{equation}\label{dist:dim21}
\begin{aligned}
y &\in W^\mathrm{u}_\mathrm{loc} (x)  &\quad &\Rightarrow  &\quad
\|Df^{n}|T_y W^\mathrm{u}_\mathrm{loc} (x)\|  &\le C_0 \|Df^{n}|E^\mathrm{u}_x\|   \, ,\\
y &\in W^\mathrm{u}_\mathrm{loc} (x)\cap K   &\quad &\Rightarrow  &\quad
\|Df^{-n}|E^\mathrm{s}_y\|   &\le C_0 \|Df^{-n}|E^\mathrm{s}_x\|   \, ,\\
y &\in W^\mathrm{s}_\mathrm{loc} (x) &\quad &\Rightarrow  &\quad
\|Df^{-n}|T_y W^\mathrm{s}_\mathrm{loc} (x)\| &\le C_0 \|Df^{-n}|E^\mathrm{s}_x\|  \, ,\\
y &\in W^\mathrm{s}_\mathrm{loc} (x)\cap K  &\quad  &\Rightarrow  &\quad
\|Df^{n}|E^\mathrm{u}_x\|   &\le C_0 \|Df^{n}|E^\mathrm{u}_x\|   \, .
\end{aligned}
\end{equation}
Using the bracket, it follows for every $x, y\in K$ such that $\tilde{\mathsf{d}}_n(x,y)<\rho_0$:
\begin{equation}\label{dist:dim22}
\|Df^{n}|E^\mathrm{u}_y\|  \le C^2_0 \|Df^{n}|E^\mathrm{u}_x\| \qand 
\|Df^{-n}|E^\mathrm{s}_y\|  \le C^2_0 \|Df^{-n}|E^\mathrm{s}_x\|\; .\end{equation}

For each $n\ge 0$, let $B_n$ be the set of points $x\in K$ such that 
\[  \|D_xf^{-n}|E^\mathrm{s}\|\ge  e^{(-\chi_s-\delta)n} \qand  \|D_xf^{n}|E^\mathrm{u}\|\le  e^{(\chi_u+\delta)n}\; .\]
Again, for $n$ large enough, by the Birkhoff ergodic Theorem we have $\mu(B_n)> 1/2$.
By the same argument as in the proof of \cref{t:example_conformal} (using \cref{c.Katok} instead of \cref{t:Katok}), 
there exists a positive number $\rho < \rho_0$ such that for all sufficiently large $n$ we can find a $(\tilde {\mathsf{d}}_n,\rho)$-separated set $F_n \subset B_n$ of cardinality at least $e^{(h_\mu - \delta)2n}$. 
As before, we use specification \cite[Thrm.~18.3.9]{KH} to shadow each $x \in F_n$ by a periodic point $y = f^{2n+2n_0}(y)$ in such a way that $\tilde {\mathsf{d}}_n(x,y) < \rho/2$, where $n_0 \ge 0$ is  independent of $n$.
Let $G_n$ be the set of periodic points $y$ obtained in this way; it has the same cardinality as $F_n$.
Since $x \in B_n$, it follows from \eqref{dist:dim22} that: 
\begin{equation}\label{e.domingo}
  \|Df^{-n-n_0}|E^\mathrm{s}_y\|\ge  e^{(-\chi_s-2\delta)n} \qand  \|Df^{n+n_0}|E^\mathrm{u}_y\|\le  e^{(\chi_u+2\delta)n}\; .\end{equation}
provided $n$ is large enough.

Let $\Pi_n$ be the union of the orbits of the points in $G_n$.

\begin{claim}
If $n$ is large enough then
the set $\Pi_n$ is $(\mathsf{d},\epsilon_n)$-separated with $\epsilon_n \coloneqq e^{-(\chi_\mu + 3\delta)(n+1)}$.
\end{claim}

\begin{proof}
Take a pair of distinct points $x$, $y \in \Pi_n$, and let us prove that $\mathsf{d}(x,y) > \epsilon_n$. If $\mathsf d(x,y) > \rho_0$ then there is nothing to prove. Otherwise, as both points are fixed by $f^{2n+2n_0}$, so, by \cref{expansivityus}, there exists $k$ with $1\le k \le n + n_0$ such that $\tilde {\mathsf d}_k(x,y)<\rho_0$ and:
\[
\mathsf d^\mathrm u(f^k(x),  f^k([x, y]))> c \cdot \rho_0\quad \text{or} \quad
\mathsf d^\mathrm s(f^{-k}(x), f^{-k}([y, x]))> c \cdot  \rho_0 \; .
\]
Let us consider the case where the first inequality holds; the other case is similar. 
Putting $z\coloneqq[x,y]$, we have:
\begin{alignat*}{2}
\mathsf{d}^\mathrm{u}(x,z) &\ge C_0^{-1} {\mathsf{d}^\mathrm{u}(f^{k}(x),f^k(z))}\cdot {\|Df^{k}|E^\mathrm{u}_{x}\|^{-1}} &\quad&\text{(by \eqref{dist:dim21})}\\
&\ge C_0^{-1} \cdot c \cdot \rho_0 \cdot \|Df^{n+n_0}|E^\mathrm{u}_{x}\|^{-1}  &\quad&\text{(since $Df|E^\mathrm{u}$ is expanding)} \\
				&>   C_0^{-1} \cdot c \cdot \rho_0  \cdot e^{-(\chi_\mu + 2\delta) n} &\quad&\text{(by \eqref{e.domingo}).}
\end{alignat*}
Since local stable and unstable manifolds are uniformly transverse, there exists a constant $C_1>0$ such that $\mathsf{d}(x,y) \ge C_1 \cdot   \mathsf{d}^\mathrm{u}(x,z)$. This implies:
\[\mathsf{d}(x,y)  \ge C_1\cdot C_0^{-1} \cdot c \cdot \rho_0  \cdot e^{-(\chi_\mu + 2\delta) n}\; .\]
It follows that $\mathsf{d}(x,y) > \epsilon_n$, for $n$ uniformly sufficiently large.
\end{proof}

The same argument as in the proof of \cref{t:example_conformal} (based on \cref{t:lower_general} again) yields that 
$\underline{\mo}(\cM_f(K)) \ge 2(h_\mu - 2\delta)/(\chi_\mu + 3\delta)$.
As $\delta$ is arbitrarily close to $0$, we conclude that $\underline{\mo}(\cM_f(K))$ is at least $2 h_\mu/\chi_\mu = d = \dim K$ by \eqref{entrop_formula_conserv}. It follows that $\mo(\cM_f^\mathrm{erg}(K)) = d$.
This completes the proof of \cref{t:example_conformal}.
\end{proof}

\section{Metric emergence and quantization of measures}\label{s:quantization}

\subsection{Quantization of measures}

The problem of \emph{quantization of measures} consists in approximating efficiently a given measure by another measure with finite support: see \cite{GrafL}.

Let $(Y,\mathsf{d})$ be a compact metric space.
Consider the set of probability measures $\cM(Y)$ endowed with  a metric also denoted $\mathsf{d}$, which can be either
a $q$-Wasserstein metric $\mathsf{W}_q$, $q \in [1,\infty)$ or the L\'evy--Prokhorov metric~$\mathsf{LP}$.

\begin{definition}\label{def:Q}  
The \emph{quantization number} of a measure $\mu \in \cM(Y)$ at a scale (or resolution) $\epsilon>0$, denoted $Q_\mu(\epsilon)$, is defined as the least integer $N$ such that there exists a probability measure $\nu$ with $\mathsf{d}(\mu,\nu) \le \epsilon$ and supported on a set of cardinality $N$. 
\end{definition}

Here is a reformulation of the definition when a Wassertein metric is used:

\begin{proposition}\label{rephrasing}
The quantization number $Q_\mu(\epsilon)$ for the $q$-Wasserstein metric $\mathsf{W}_q$ is the minimal cardinality $N$ of a set  $F = \{x_1, \dots, x_N\}$ so that:
\[\int_Y \left(\mathsf{d}(x, F)\right)^q d\mu(x) \leq \epsilon^q\; .\]
\end{proposition}

\begin{proof}
Fix $\epsilon>0$ and let $F \subset Y$ be a set of  minimal cardinality $N$ such that $\int [\mathsf{d}(x,F)]^q d\mu(x) \le \epsilon^q$. 

Take a measurable map $h \colon Y \to F$ that associates to each element in $Y$ a closest element in $F$ (w.r.t.\ the $\mathsf{d}$ metric). Let $\nu \coloneqq h_* \mu \in \cM(Y)$; this is a measure supported on $F$. We claim that $\mathsf{W}_q (\mu,\nu) \le  \epsilon$. Indeed, $\pi \coloneqq (\id \times h)_* (\mu)$ is a transport plan from $\mu $ to $ \nu$ with cost 
$$
\int [\mathsf{d}(x,h(x))]^q d\mu(x) = \int [\mathsf{d}(x,F)]^q d\mu(x) \leq \epsilon^q \, .
$$
We have shown that $Q_{\mu}(\epsilon) \le N$.

Let us prove the reverse inequality. 
Let $\nu \in \cM(Y)$ be a measure whose support $F' \subset Y$ has cardinality $Q_{\mu}(\epsilon)$ and such that $\mathsf{W}_q({\mu},\nu) \le  \epsilon$. This means that there is a transport plan $\pi \in \cM(Y \times Y)$ from ${\mu}$ to $\nu$ with cost at most $\epsilon^q$. 
Consider a disintegration of $\pi$, that is, a family $(\nu_\xi)$ of elements of $\cM(Y)$, defined for ${\mu}$-almost every $\xi \in Y$, such that $\pi = \int \delta_\xi \otimes \nu_\xi d{\mu}(\xi)$. As the second marginal of $\pi$ equals $\nu$, whose support is the finite set $F'$, it follows that $\supp \nu_\xi \subset F'$ for ${\mu}$-almost every $\xi$.
Therefore:
$$
\epsilon^q \geq \textrm{cost}(\pi)
=   \iint [\mathsf{d}(\xi,\eta)]^q d\nu_\xi(\eta) d\mu(\xi) 
\ge \int [\mathsf{d} (\xi,F')]^q d \mu(\xi)  
 \, .
$$
This shows that $N \le \card F'= Q_{\mu}(\epsilon)$.
\end{proof}

Here is a similar characterization of the quantization number for the case of the L\'evy--Prokhorov metric:

\begin{proposition}\label{rephrasing_LP}
The quantization number $Q_{\mu}(\epsilon)$ for the $\mathsf{LP}$ metric is the least number of closed balls of radius $\epsilon$ that cover a set of $\mu$-measure at least $1-\epsilon$.
\end{proposition}

\begin{proof}
Straightforward.
\end{proof}

Similarly to the definition of the lower and upper box-counting dimensions, following \cite[p.~155]{GrafL} the \emph{lower} and \emph{upper quantization dimensions} of $\mu\in \cM(Y)$ 
are defined as:
\[
\underline\dim(\mu)  \coloneqq  \liminf_{\epsilon \to 0} \frac{\log Q_\mu(\epsilon)}{-\log \epsilon} \qand
\overline \dim(\mu)  \coloneqq  \limsup_{\epsilon \to 0} \frac{\log Q_\mu(\epsilon)}{-\log \epsilon}\; .
\]
If these numbers coincide then they are denoted by $\dim(\mu)$ and called \emph{quantization dimension}.
Furthermore, the \emph{lower} and \emph{upper quantization orders} are defined as:
\[
\underline \qo(\mu)  \coloneqq  \liminf_{\epsilon \to 0} \frac{\log\log Q_\mu(\epsilon)}{-\log \epsilon} \qand
\overline \qo(\mu)  \coloneqq  \limsup_{\epsilon \to 0} \frac{\log\log Q_\mu(\epsilon)}{-\log \epsilon}\; .
\]
If these numbers coincide then they are denoted by $\qo(\mu)$ and called \emph{quantization order}.

\begin{proposition}\label{l:quant_cover} 
For any resolution $\epsilon>0$, the quantization number of any $\mu \in \cM(Y)$ is bounded from above by the covering number of~$Y$, 
that is:
$$
Q_\mu(\epsilon) \le D_Y(\epsilon) \, .
$$
In particular,
\begin{alignat*}{3}
\underline \dim(\mu) &\le \underline \dim(Y) &\quad&\text{and}&\quad \overline \dim(\mu) &\le \overline \dim(Y) \, , 
\\
\underline \qo(\mu) &\le \underline \mo(Y) &\quad&\text{and}&\quad \overline \qo(\mu) &\le \overline \mo(Y)\, .	
\end{alignat*}
\end{proposition}

\begin{proof}
Given an $\epsilon$-dense set $F$ of cardinality $N$, we can transport any measure $\mu \in \cM(Y)$ to a measure supported on $F$ with cost $\leq \epsilon^q$ with respect to the cost function $\mathsf{d}^q$.  This shows that  $Q_\mu(\epsilon) \le  D_Y(\epsilon)$ with respect to the $\mathsf{W}_q$ distance. 
In view of \cref{rephrasing_LP}, the same statement is also immediate for the $\mathsf{LP}$ distance. Then it follows that quantization dimensions are bounded by box-counting dimensions, and quantization orders are bounded by metric orders.
\end{proof}

\begin{example}\label{quantization 1d} 
Consider $Y = [0,1]$ with the usual metric, and endow the space $\cM([0,1])$ with the metric $\mathsf{W}_q$.
Consider the Lebesgue measure on $[0,1]$; its quantization number is:
$$
Q_{\Leb}(\epsilon) = \left\lceil \frac{1}{2 (q+1)^{1/q} \, \epsilon} \right\rceil \, ,
$$
and in particular the quantization dimension is $1$.
Indeed, given $N\ge 1$, the probability measure on $[0,1]$ supported on $N$ points which is $\mathsf{W}_q$-closest to Lebesgue is: 
$$
\nu_N \coloneqq \frac{1}{N}\sum_{j=1}^N \delta_{\frac{2j-1}{2N}} \, ,
\quad\text{for which} \quad
\mathsf{W}_q(\nu_N, \Leb) = \frac{1}{2 (q+1)^{1/q} \, N} 
$$
(see \cite[p.~69]{GrafL}), so the asserted formula for $Q_{\Leb}(\epsilon)$ follows.
\end{example}

\begin{example}\label{quantization higher d}
If $\mu$ is a compactly supported measure on $\R^d$ which is absolutely continuous with respect to Lebesgue measure then $\dim(\mu) = d$; 
actually there is a precise asymptotic formula for the quantization number $Q_\mu(\epsilon)$  with respect to the $\mathsf{W}_q$ distance: see \cite[p.~78, p.~52]{GrafL}.
\end{example}

\begin{example}
See the paper \cite{LM} for the computation of the quantization dimension of certain self-similar measures ($F$-conformal measures) supported on fractal sets defined by conformal iterated function systems; let us note that the answer depends on the exponent $q$.
\end{example}

\begin{example}
The metric entropy of an ergodic measure can be described in terms of quantization numbers: see \cref{ss:entropy_quantization}.
\end{example}

In this paper, we are mostly interested in the situation where the quantization orders are positive, and so the quantization dimensions are infinite.

In view of \cref{l:quant_cover}, the next result yields measures with maximal quantization order:

\begin{otherthm}\label{t:fat_measure}
Let $Y$ be a Borel subset of a compact metric space $Z$. 
Then there exists a probability measure $\mu \in \cM(Y)$ such that:
$$\underline \qo (\mu)= \underline{\mo}(Y) \qand 
\overline \qo (\mu) = \overline{\mo}(Y) \, .
$$
\end{otherthm}

The proof is given in \cref{ss:fat_measure}.

\subsection{Ergodic decomposition} 

Let $X$ be a compact metric space and let $f \colon X \to X$ be a continuous map. 
Recall that the empirical measure at a point $x \in X$ is defined as $\mathbf{e}^f(x) \coloneqq \lim \frac{1}{n} \sum_{i=0}^{n-1} \delta_{f^i x}$, 
when this limit exists.
By the ergodic decomposition theorem (see \cite[\S~13]{DGS} or \cite[\S~II.6]{Mane}), there exists a Borel set $X_0 \subset X$ with full probability (that is, $\mu(X_0)=1$ for every $\mu \in \cM_f(X)$) such that for every $x\in X_0$, the empirical measure $\mathbf{e}^f(x)$ is $f$-invariant and ergodic. 
So for any $\mu \in \cM_f(X)$, the measure $\mathbf{e}^f_* \mu \in \cM(\cM(X))$ gives full weight to the set $\cM_f^\mathrm{erg}(X) \subset \cM_f(X)$ of ergodic measures, and its barycenter $\mathrm{bar}(\mathbf{e}^f_* \mu) \coloneqq \int \nu \, d(\mathbf{e}^f_* \mu)(\nu)$ is $\mu$. The probability measure $\mathbf{e}^f_* \mu$ is called the \emph{ergodic decomposition} of $\mu$.
There is a canonical bijection $\cM_f(X) \to \cM(\cM_f^\mathrm{erg}(X))$, namely $\mu \mapsto \mathbf{e}^f_* \mu$.

\begin{remark}\label{r:semicont_erg_dec}  
Generic conservative diffeomorphisms (in any topology) constitute continuity points of the ergodic decomposition of Lebesgue measure: see \cite[Thrm.~B]{AB}.
We will see later in \cref{ss:KAM} non-trivial examples of continuity points w.r.t.\ the $C^\infty$ topology.
\end{remark}

Let us note the following property for later use:

\begin{lemma}[Factors and ergodic decompositions] \label{l:push_erg}
Suppose $X$ and $Y$ are compact metric spaces and let 
$f \colon X \to X$ and $g \colon Y \to Y$ be continuous maps which are semi-conjugate ($g \circ \phi = \phi \circ f$) via a continuous $\phi \colon X \to Y$. 
Let $\Phi \colon \cM(X) \to \cM(Y)$ be the map $\mu \mapsto \phi_*\mu$. 
Then $\Phi(\cM_f(X)) \subset \cM_g(Y)$, 
$\Phi(\cM_f^\mathrm{erg}(X)) \subset \cM_g^\mathrm{erg}(Y)$, and:
$$
\forall \mu \in \cM_f(X), \quad \mathbf{e}^g_*(\phi_* \mu) = \Phi_*(\mathbf{e}^f_*(\mu)) \, .
$$
\end{lemma}

When no confusion arises, we will write $\phi_*$ instead of $\Phi$, so the last equation becomes $\mathbf{e}^g_*(\phi_* \mu) = \phi_{**}(\mathbf{e}^f_*(\mu))$.

\begin{proof}
Let $\mu \in \cM_f(X)$ and let $\nu \coloneqq \phi_* (\mu)$.
Then $g_*\nu = (g\circ \phi)_*(\mu) = (\phi \circ f)_*(\mu) = \nu$, that is, $\nu \in \cM_g(Y)$, proving the first assertion.

Note that that if $B \subset Y$ is a $g$-invariant Borel set then $\phi^{-1}(B)$ is $f$-invariant; it follows that $\nu$ is ergodic if $\mu$ is, proving the second assertion.

Let $\hat\mu \coloneqq \mathbf{e}^f_*(\mu)$ and $\hat\nu \coloneqq \mathbf{e}^g_*(\nu)$ be the corresponding ergodic decompositions.
For every Borel set $B \subset Y$, we have:
\begin{multline*}
\nu(B) =
\mu(\phi^{-1}(B)) = 
\int_{\cM(X)} \eta(\phi^{-1}(B)) \, d \hat\mu (\eta)  \\=
\int_{\cM(X)} (\Phi(\eta))(B) \, d \hat\mu (\eta)  =
\int_{\cM(Y)} \xi(B) \, d(\Phi_*(\hat\mu))(\xi) \, .
\end{multline*}
This means that $\nu$ is the barycenter of $\Phi_*(\hat\mu)$.
Since $\hat\mu$ gives full weight to $\cM_f^\mathrm{erg}(X)$, the measure   
$\Phi_*(\hat\mu)$ gives full weight to $\cM_g^\mathrm{erg}(Y)$, and by uniqueness of the ergodic decomposition, it follows that $\Phi_*(\hat\mu)$ equals $\hat\nu$, the ergodic decomposition of~$\nu$.
\end{proof}

\subsection{Metric emergence}

Given a continuous self-map $f \colon X \to X$ of a compact metric space $X$, 
we consider the set $\cM(X)$ with a metric $\mathsf{d} \in \{\mathsf{W}_p \st 1\le p<\infty\} \cup \{\mathsf{LP}\}$.
We have introduced in \cref{def.metric_em} the \emph{metric emergence} of a measure $\mu \in \cM(X)$.
In the case $\mu$ is invariant, we have the following characterization of metric emergence:

\begin{proposition}\label{p:emergence_Q}
For every dynamics $f \colon X \to X$, 
the metric emergence of any invariant measure $\mu \in \cM_f(X)$ equals the quantization number of the ergodic decomposition $\hat\mu \coloneqq \mathbf{e}^f_* \mu$ (considered as a measure on $\cM(X)\,$):
$$
\Eme_\mu(f)(\epsilon) = Q_{\hat\mu}(\epsilon) \, ,
$$
where $Q_{\hat\mu}$ is the quantization number of $\hat \mu$ for the metric $\mathsf{W}_1$ of $\cM(\cM(X))$.
\end{proposition}

\begin{proof}
Combine \cref{def.metric_em,rephrasing}.
\end{proof}

\begin{remark}
Given a parameter $q\ge 1$, we may define the  \emph{$q$-emergence}
of an $f$-invariant measure $\mu$  at scale $\epsilon>0$ as:
$$
\Eme_\mu^{(q)}(\epsilon) \coloneqq 
\min \left\{ N \st \exists F \subset \cM(X) \text{ with } \card F \le N, \  \int \mathsf{d}(\mathbf{e}^f(x),F)^q d\mu(x) \le \epsilon^q \right\} .
$$
By \cref{rephrasing}, $q$-emergence is the quantization number of the ergodic decomposition with respect to the metric $\mathsf{W}_q$ on $\cM(\cM(X))$. For simplicity we will focus our study on $q=1$. 
\end{remark}

Metric and topological emergences may be compared as follows:

\begin{proposition}\label{compa metric_topo emergence}
For every dynamics $f \colon X \to X$, the metric emergence of any invariant measure $\mu \in \cM_f(X)$ is at most the topological emergence:
\[
\Eme_\mu(f)(\epsilon)\le \Eme_{\mathrm{top}}(f)(\epsilon)\; ,\quad \forall \epsilon>0\; ,
\]
provided both emergences are computed using the same metric~$\mathsf{W}_p$ or $\mathsf {LP}$ on $\cM(X)$.
\end{proposition}  
\begin{proof}
By \cref{p:emergence_Q}, the metric emergence $\Eme_\mu(f)(\epsilon)$ equals the quantization number $Q_{\hat\mu}(\epsilon)$ of the ergodic decomposition $\hat\mu \coloneqq \mathbf{e}^f_* \mu$. Note that $\hat\mu$ is a measure on   $Y \coloneqq \cM_f^{\mathrm{erg}}(X)$ which is a Borel subset of $Z\coloneqq\cM(X)$. By \cref{l:quant_cover}, $Q_{\hat\mu}(\epsilon)$ is at most the relative covering number $D_Z(Y,\epsilon)$, which equals the topological emergence $\Eme_{\mathrm{top}}(f)(\epsilon)$ by its own definition~\eqref{e.def_top_emergence}.
\end{proof}

We are now able to deduce the variational principle for emergence announced at the introduction:

\begin{proof}[Proof of \cref{t:var_princ}] 
Applying \cref{t:fat_measure} with $Y\coloneqq \cM_f^\mathrm{erg}(X)$, $Z\coloneqq \cM (X)$ and $q=1$,
we obtain a probability measure $\nu \in \cM(\cM_f^\mathrm{erg}(X))$ such that: 
\[
\underline {\qo}(\mu)  = \underline {\mo}(\cM_f^\mathrm{erg}(X))\qand 
\overline {\qo}(\mu)   = \overline {\mo}(\cM_f^\mathrm{erg}(X))\; .
\]	
Let $\mu \coloneqq \int_{\cM(X)}\eta d\nu(\eta)$. Since $\nu$ gives full weight to $\cM_f^{\mathrm{erg}}(X)$, the measure $\mu$ is invariant and its ergodic decomposition is $\nu$. 
Bearing in mind \cref{p:emergence_Q} and the definitions of lower and upper quantizations orders and metric orders, we obtain the equalities stated in \cref{t:var_princ}.
\end{proof}

\subsection{Some properties of quantization numbers}\label{ss:quant_properties}

In this \lcnamecref{ss:quant_properties} we prove a few general properties about quantization numbers that will be needed later. 
To simplify matters, \emph{all quantization numbers in this \lcnamecref{ss:quant_properties} are computed w.r.t.\ the $\mathsf{W}_1$ metric.}

\begin{lemma}\label{l:continuity}
For all $\mu_1$, $\mu_2 \in \cM(Y)$,
$$
\mathsf{W}_1(\mu_1,\mu_2) \le \epsilon
\quad \Rightarrow \quad Q_{\mu_2}(2\epsilon) \le Q_{\mu_1}(\epsilon) \, .
$$
\end{lemma}

\begin{proof}
Immediate. 
\end{proof}

The next two lemmas deal with pushing forward a measure under a Lipschitz map, and the effect of this operation on the quantization numbers:

\begin{lemma}\label{l:push_Lip}
Let $f \colon (Y,\mathsf{d}) \to (Z,\mathsf{d})$ be a $\kappa$-Lipschitz map between compact metric spaces.
Let $F \colon (\cM(Y),\mathsf{W}_1) \to (\cM(Z),\mathsf{W}_1)$ be the map $\mu \mapsto f_* \mu$.
Then $F$ is $\kappa$-Lispchitz.
\end{lemma}

\begin{proof}
Given $\mu_1$, $\mu_2 \in \cM(Y)$, consider a transport plan $\pi \in \cM(Y \times Y)$.
Then $\tilde \pi \coloneqq (f \times f)_*(\pi)$ is a transport plan from $f_*\mu_1$ to $f_*\mu_2$ with:
$$
\mathrm{cost}(\tilde \pi) = \int \mathsf{d}(f(x),f(y))\, d\pi(x,y) \le \kappa \int \mathsf{d}(x,y) \, d\pi(x,y)  = \kappa \, \mathrm{cost}(\pi) \, .
$$
So $\mathsf{W_1}(f_*\mu_1 , f_*\mu_2) \le \kappa \mathsf{W_1}(\mu_1 , \mu_2)$.
\end{proof}

\begin{lemma}\label{l:quant_Lip} 
Let $(Y,\mathsf{d})$ and $(Z,\mathsf{d})$ be compact metric spaces.
Let $f \colon Y \to Z$ be a $\kappa$-Lipschitz map.
Given a measure $\mu \in \cM(Y)$, consider its push-forward $\nu \coloneqq f_*\mu \in \cM(Z)$.
Then for every $\epsilon>0$, we have:
\[
Q_\mu (\epsilon) \ge Q_\nu (\kappa\epsilon) \, .
\]
\end{lemma}

\begin{proof}
Given $\mu \in \cM(Y)$ and $\epsilon>0$, let $\tilde \mu \in \cM(Y)$ be a measure supported on $n \coloneqq Q_\mu(\epsilon)$ points with $\mathsf{W}_1(\mu, \tilde\mu) \le \epsilon$.
By \cref{l:push_Lip}, the measures $\nu \coloneqq f_*\mu$ and $\tilde \nu \coloneqq f_*\tilde \mu$ satisfy $\mathsf{W}_1(\nu, \tilde\nu) \le \kappa\epsilon$. Since $\tilde \nu$ is supported on at most $n$ points, we conclude that $Q_\nu (\kappa\epsilon) \le n$.
\end{proof}

The next two lemmas will be used several times, in particular in the proof of \cref{t:fat_measure}:

\begin{lemma}\label{l:submeasure}
Let $\mu$, $\mu_1\in \cM(Y)$ be such that $\mu \ge t \mu_1$, for some $t>0$.
Then:
$$
Q_{\mu}(t \epsilon) \ge  Q_{\mu_1}(\epsilon) \, .
$$
\end{lemma}

\begin{proof}
Let $\tilde \epsilon\coloneqq  t \epsilon$. 
Let $\nu$ be a measure supported on a set of cardinality $\ell \coloneqq Q_{\mu}(\tilde\epsilon)$ and 
such that $\mathsf{W}_1(\mu,\nu) \leq \tilde\epsilon$.
Let $\pi$ be a transport plan from $\mu$ to $\nu$ with cost (w.r.t.\ $\mathsf{d}$) not greater than $\tilde \epsilon$.

The Radon--Nikodym derivative $f \coloneqq \frac{d \mu_1}{d \mu}$ is well-defined and satisfies $0 \le f \le t^{-1}$ at $\mu$-a.e.\ point.
Consider the measure $\tilde\pi$ on $Y \times Y$ defined by:
$$
d\tilde\pi(x,y) = f(x) d\pi(x,y) \, .
$$
Then $\tilde\pi$ is a probability, its first marginal is $\mu_1$, and its second marginal is some measure $\tilde\nu$ which is absolutely continuous with respect to $\nu$ and therefore supported on a set of cardinality at most $\ell$.
We have:
\begin{multline*}
\mathsf{W}_1(\mu_1,\tilde\nu)\le \mathrm{cost}(\tilde\pi) = \int \mathsf{d}(x,y) \, d\tilde\pi(x,y)  = 
\int \mathsf{d}(x,y) \, f(x) \, d\pi(x,y) 
\\
\le t^{-1} \int \mathsf{d}(x,y) \, d\pi(x,y) = 
t^{-1} \mathrm{cost}(\pi) \leq t^{-1} \tilde \epsilon
 \, .
\end{multline*}
That is, $\mathsf{W}_1(\mu_1,\tilde\nu) \leq t^{-1} \tilde\epsilon = \epsilon$.
It follows that $\card \supp \tilde \nu\ge  Q_{\mu_1}(\epsilon)$, and so $\ell \ge  Q_{\mu_1}(\epsilon)$, as claimed.
\end{proof}

\begin{lemma}\label{l:cost}
Let $\epsilon > 0$ and let $F \subset Y$ be an $\epsilon$-separated set.
Let $n \coloneqq \card F$ and let $\mu$ be the equidistributed probability measure with support $F$.
Let $\nu$ be any probability measure whose support has cardinality $m < n$.
Then:
$$
\mathsf{W}_1(\mu,\nu) \ge  \frac{n-m+1}{n}  \cdot \frac{\epsilon}{2} \, .
$$
\end{lemma}

\begin{proof}
Let  $\supp \mu = \{x_1, \dots, x_n \}$ and  $\supp \nu = \{y_1, \dots, y_m \}$.
Since $\mu$ is equidistributed, transport plans from $\mu$ to $\nu$ take the form:
$$
\pi = \pi_A =  \frac{1}{n} \sum_{i=1}^n \sum_{j=1}^m a_{ij} \delta_{(x_i,y_j)} \, ,
$$
where $A = (a_{ij})$ is a row-stochastic $n \times m$ matrix (that is, each $a_{ij}$ is non-negative and $\sum_{j=1}^m a_{ij} = 1$ for every $i$). 
The cost of $\pi_A$  is:
$$
\mathrm{cost}(\pi_A) = \frac{1}{n} \sum_{i=1}^n \sum_{j=1}^m a_{ij} \mathsf{d}(x_i,y_j) \, ,
$$
which can be viewed as an affine function on the set of row-stochastic matrices.
This set is compact and convex, and its extremal points consist on the matrices that contain exactly one entry equal to $1$ on each row.
So it is sufficient to consider matrices of this type in order to find a lower bound for the cost.
Thus consider a row-stochastic matrix $A = A_T$ whose nonzero entries are $a_{i,T(i)} = 1$ for some map $T \colon \{1,\dots,n\} \to \{1,\dots,m\}$.

\begin{claim}
For every $j \in \{1,\dots, m\}$ such that $s \coloneqq \card T^{-1}(j) \ge 2$, the following holds:
\begin{equation}\label{e:claim_estimate}
\sum_{i \in T^{-1}(j)} \mathsf{d}(x_i,y_j) \ge   \frac{s\cdot \epsilon}{2}\, .
\end{equation}
\end{claim}

\begin{proof}[Proof of the claim] Indeed, 
write $T^{-1}(j) = \{i_1, \dots, i_s\}$; then the left hand side of \eqref{e:claim_estimate} equals:
$$
\sum_{k=1}^s \mathsf{d}(x_{i_k},y_j)
= \frac{1}{s-1} \sum_{1\le k<\ell\le s} \big[\mathsf{d}(x_{i_k},y_j) + \mathsf{d}(x_{i_\ell},y_j)\big] \, . 
$$
For every $1\le k < \ell\le s$, since $F$ is $\epsilon$-separated, it hold:
\[
 \mathsf{d}(x_{i_k},y_j) + \mathsf{d}(x_{i_\ell},y_j) \ge \mathsf{d}(x_{i_k}, x_{i_\ell}) \ge \epsilon \, .
\]
So we obtain:
$$
\sum_{k=1}^s \mathsf{d}(x_{i_k},y_j)
\ge \frac{1}{s-1} \cdot \frac{s(s-1)}{2} \cdot \epsilon  
= \frac{s \cdot \epsilon}{2} \, ,
$$
as claimed. 
\end{proof}

Using \eqref{e:claim_estimate}, we estimate:
$$
\mathrm{cost}(\pi_{A_T}) 
=    \frac{1}{n} \sum_{i=1}^n \mathsf{d}(x_i,y_{T(i)})
=    \frac{1}{n} \sum_{j=1}^m \sum_{i \in T^{-1}(j)} \mathsf{d}(x_i,y_j)
\ge  \frac{n_*}{n}  \frac{\epsilon}{2}  \, ,
$$
where
\begin{align*}
n_* 
\coloneqq \sum_{\substack{j\in \{1,\dots,m\}, \\ \card T^{-1}(j)> 1}} \card T^{-1}(j)
&=  n - \sum_{\substack{j\in \{1,\dots,m\}, \\ \card T^{-1}(j)\le 1}} \card T^{-1}(j) \\
&= n - \card \big\{j\in \{1,\dots,m\} \st \card T^{-1}(j) = 1 \big\} \\
&=   n - m + \card \big\{j\in \{1,\dots,m\} \st \card T^{-1}(j) \neq 1 \big\} \\
&\ge n - m + 1 \, ,
\end{align*}
since $m<n$.
We conclude that $\mathrm{cost}(\pi_A)$ is at least $\frac{n-m+1}{n}   \frac{\epsilon}{2} $ for every matrix $A$ of type $A_T$, and therefore for every row-stochastic matrix $A$.
The \lcnamecref{l:cost} follows.
\end{proof}

\subsection{Existence of a measure with essentially maximal quantization numbers}\label{ss:fat_measure}

In this subsection we prove \cref{t:fat_measure}, which was used to deduce \cref{t:var_princ}.

\begin{proof}[Proof of \cref{t:fat_measure}]\label{proof:t:fat_measure}
It is sufficient to prove the \lcnamecref{t:fat_measure}  assuming that $\cM(Y)$ is metrized with the $\mathsf{W}_1$ distance. 
Indeed, by the first inequality in \eqref{Wp_monotone} (see p.~\pageref{compaWpLP}), if the exponent $q$ is reduced then the metric $\mathsf{W}_q$ does not increase, and so neither do quantization numbers and orders. 
Furthermore, by the second inequality in \eqref{compaWpLP}, the metric $\mathsf{LP}$ is bounded from below by a constant factor of the metric $\mathsf{W}_1$, and so quantization numbers and orders with respect to  $\mathsf{LP}$ are bounded from below by the corresponding quantities with respect to $\mathsf{W}_1$.
So from now on we assume that $\cM(Y)$ is metrized with the $\mathsf{W}_1$ distance.

By \cref{l:quant_cover}, it is sufficient to show the existence of a measure $\mu\in \cM(Y)$ such that:
\begin{equation}\label{suffi:t:fat_measure} \underline \qo(\mu) \ge \underline \mo(Y)\qand  \overline \qo(\mu) \ge \overline \mo(Y)\, .\end{equation}

Recall that, given $\epsilon>0$, the corresponding packing number is denoted by $S_Y(\epsilon)$. We set $\epsilon_i\coloneqq  2^{-i^2}$ for every $i\ge 1$. Let $F_i \subset Y$ be a $4\epsilon_i$-separated set of cardinality $n_i\coloneqq  S_Y(4 \epsilon_i)$, and let $\mu_i \in \cM(Y)$ be the equidistributed probability measure with support $F_i$.
By \cref{l:cost}, if $\nu$ is a probability measure whose support has cardinality at most  $m_i \coloneqq \lceil n_i/2 \rceil$ then
$$
\mathsf{W}_1(\mu_i, \nu) \ge \epsilon_i \, .
$$
That is, in terms of quantization number:
\begin{equation}\label{Qmui:mi} 
	Q_{\mu_i}(\epsilon_i)  \ge m_i \, .
\end{equation}

Now consider the following probability measure:
$$
\mu \coloneqq \sum_{i=1}^\infty t_i \mu_i \, , \quad \text{where} \quad  t_i \coloneqq 2^{-i} \, .
$$
By \cref{l:submeasure}, for every $i\ge 1$ we have $Q_{\mu}(\tilde\epsilon_i) \ge Q_{\mu_i}( \epsilon_i)$, where $\tilde\epsilon_i \coloneqq t_i \epsilon_i$.
Using \eqref{Qmui:mi} we obtain:
\begin{equation}\label{ine:varia:principle}
  \frac{\log\log Q_\mu(\tilde\epsilon_i)}{-\log \tilde\epsilon_i}\ge   \frac{\log\log m_i}{-\log \tilde \epsilon_i}
\mathrel{\underset{i \to \infty}{\scalebox{1.7}{$\sim$}}}
  \frac{\log\log n_i}{-\log (4\epsilon_i)} \coloneqq  \frac{\log \log S_Y(4 \epsilon_i)}{-\log (4 \epsilon_i)}
\;  .
\end{equation}

\begin{claim}\label{lemm:cluster}
The following equalities hold:
\begin{alignat}{2}
\liminf_{i \to \infty} \frac{\log \log Q_\mu(\tilde \epsilon_i)}{-\log \tilde \epsilon_i} &= \underline{\qo}(\mu)\; ,  &\quad 
\limsup_{i \to \infty} \frac{\log \log Q_\mu(\tilde \epsilon_i)}{-\log \tilde \epsilon_i} &=\overline{\qo}(\mu)\; , \label{e.cluster_qo}
\\
\liminf_{i \to \infty} \frac{\log \log S_Y(4 \epsilon_i)}{-\log (4 \epsilon_i)}&= \underline{\mo}(Y)\; ,  &\quad 
\limsup_{i \to \infty} \frac{\log \log S_Y(4\epsilon_i)}{-\log (4 \epsilon_i)}&=\overline{\mo}(Y)\; . \label{e.cluster_mo}
\end{alignat}
\end{claim}

\begin{proof}[Proof of the \lcnamecref{lemm:cluster}]
Let us prove \eqref{e.cluster_mo}; the proof of \eqref{e.cluster_qo} is essentially the same.	
Given $\epsilon>0$, let $i$ be such that $\epsilon \in [4\epsilon_{i+1}, 4\epsilon_{i}]$. We have $S_Y(4\epsilon_{i}) \le S_Y(\epsilon)\le S_Y(4\epsilon_{i+1})$ and so:
$$
\frac{\log \log S_Y(4\epsilon_{i+1})}{-\log (4\epsilon_{i})}\ge \frac{\log \log S_Y(\epsilon)}{-\log \epsilon}\ge \frac{\log \log S_Y(4\epsilon_{i})}{-\log (4\epsilon_{i+1})}\; .
$$
Since $\log(4\epsilon_{i}) \sim \log(4\epsilon_{i+1})$ as $i \to \infty$, inequalities \eqref{e.cluster_mo} follow.
\end{proof}

Combining \eqref{ine:varia:principle} with \cref{lemm:cluster} we obtain inequality \eqref{suffi:t:fat_measure} and the \lcnamecref{t:fat_measure}. 
\end{proof}

\section{Examples of conservative dynamics with high metric emergence}\label{section:metric:emer}

We are going to study the emergence of dynamics on the annulus $\A$:
$$
 \A \coloneqq \cercle \times [0,1] \quad \text{with } \cercle \coloneqq \R/\Z \, .
$$
Lebesgue measure on either of theses sets is denoted by $\Leb$.

The \emph{horizontal flow} 
associated to a $C^\infty$ function $\omega \colon [0,1]\to \R$ is defined as:
\begin{equation}\label{e.horizontal_flow}
R^t_\omega \colon (\theta, \rho)\in \A\mapsto (\theta + \omega(\rho)t, \rho)\in \A \; .
\end{equation}
So $(R^t_\omega)_t$ is a conservative smooth flow on the annulus.
Assume that $\omega$ has no critical points.
Then, for every fixed $t\neq 0$, Lebesgue almost every $\rho \in [0,1]$ has the property that $\omega(\rho)\cdot t$ is irrational, and therefore for every $\theta \in \cercle$, the empirical measure $\mathbf{e}^{R^t_\omega}(\theta, \rho)$ equals: 
$$
\lambda_\rho \coloneqq \Leb_{\cercle} \otimes \delta_\rho \quad \text{(Lebesgue measure on the circle $\cercle \times \{\rho\}$).}
$$
Hence the ergodic decomposition of the Lebesgue measure with respect to the time $t$ map $R^t_\omega$ does not depend on $t \neq 0$ and is given by: 
\begin{equation}\label{e.erg_dec_Rt}
\mathbf{e}_*^{R^t_\omega}(\Leb) = 
\int_{0}^1 \delta_{\lambda_\rho} \, d\rho \, .
\end{equation}

\subsection{Robust examples of at least polynomial emergence}\label{ss:horizontal}

\begin{proposition}\label{emergence poly}
Suppose $\omega \colon [0,1]\to \R$ is a smooth function  without critical points
and let $(R^t_\omega)_t$ be the corresponding horizontal flow.
For every $t\neq 0$, the metric emergence of the time $t$ map $R^t_\omega$ 
with respect to the Wasserstein metric $\mathsf{W}_1$ is:
\[
\Eme_\Leb(R^t_\omega) (\epsilon) =
\lceil (4\epsilon)^{-1} \rceil \; . 
\]
\end{proposition}

\begin{proof}
As seen in \eqref{e.erg_dec_Rt}, the ergodic decomposition $\hat{\mu} \coloneqq \mathbf{e}_*^{R^t_\omega}(\Leb)$ is equidistributed on the curve $\{\lambda_\rho: \rho\in [0,1]\}$.
This curve endowed with the Wasserstein metric $\mathsf{W}_1$ is isometric to the unit interval $[0,1]$ endowed its usual distance; the isometry sends the measure $\hat{\mu}$ to the Lebesgue measure on $[0,1]$. 
Thus
 $Q_{\hat{\mu}}(\epsilon) =  Q_{\Leb|[0,1]}(\epsilon) = \lceil (4\epsilon)^{-1} \rceil$, by \cref{quantization 1d} with $q=1$. Using \cref{p:emergence_Q} we conclude.
\end{proof}

KAM theory ensures that most of the invariant circles of $R^t_\omega$ persist for any conservative $C^\infty$ perturbation. As a consequence, we obtain $C^\infty$-open sets of conservative surface diffeomorphisms whose metric emergence is at least of the order of $\epsilon^{-1}$: see \cref{ss:KAM}, more specifically \cref{c:robust_poly}.

\subsection{Construction of a smooth conservative flow with high emergence at a given scale}\label{ss:main_construction}

The heart of the proof of \cref{t:main} is the following result:
 
\begin{proposition}\label{p:main}
There exists $C>0$ such that for every $\epsilon_*>0$, there exists a smooth conservative diffeomorphism $h$ of $\A$ satisfying the following property.
For every function $\omega \in C^\infty([0,1], \R)$  
without critical points
and for every $t\neq 0$, the map $\Psi^t \coloneqq h \circ R^t_\omega \circ h^{-1}$ satisfies:
\[
\Eme_{\Leb}(\Psi^t) (\epsilon_*) \ge \exp(C \epsilon_*^{-2}) \, ,
\]
where the emergence is computed with respect to the Wasserstein metric $\mathsf{W}_1$. Furthermore, $h$ equals identity on a neighborhood of  the boundary of $\A$.
\end{proposition}

The proof of the \lcnamecref{p:main} will occupy the rest of this subsection. 

\begin{proof}
We will actually construct a sequence $h_n$ of diffeomorphisms such that the corresponding flows $\Psi^t_n \coloneqq h^{-1}_n \circ R^t_\omega \circ h_n$ have high emergence at a certain scale $\epsilon_n$; then we will show that for every $\epsilon_*>0$ we can choose an appropriate $h=h_n$ and obtain the conclusion of \cref{p:main}.
The proof is divided into several steps.

\medskip 
\noindent\textbf{Zeroth step.}
Let $n \ge 3$ be an arbitrary integer.
We will fix several numbers depending on $n$.
Let $N \coloneqq 32\cdot  n^2$. 
Let $M = m\cdot n$ be the multiple of $n$ as big as possible such that:
\begin{equation}\label{e.def_M}
M \le (2N)^{-1/2}e^{\pi N/4^3} \, .
\end{equation}
 It is clear from this definition that: 
\begin{equation}\label{e.magnitude_M}
\log M \asymp  n^2 \, .
\end{equation}
Finally, let $\eta \coloneqq 1/(1000 n)$ and $\kappa \coloneqq 1-\eta$.

\medskip 
\noindent\textbf{First step.}
The real proof begins with the construction of certain families of boxes in the annulus $\A$. 
An \emph{$a \times b$-box} is a set of the form $I \times J$ where $I \subset \cercle$ and $J \subset [0,1]$ are closed intervals of respective lengths $a$ (the  \emph{width} of the box) and $b$ (the \emph{height} of the box).
An \emph{$a$-square} is an $a \times a$-box.
A \emph{$k \times \ell$-family} is a disjoint collection of boxes of the form $I_i \times J_j$ where $1 \le i \le k$, $1 \le j \le \ell$. Such a family can be partitioned (in the obvious way) into $k$ subfamilies called \emph{columns} and into $\ell$ subfamilies called \emph{rows}.

Let $\cG$ be a  $8n \times 4n$-family of $\frac{1}{10 n}$-squares
contained in the lower half-annulus $\cercle \times [0,\frac{1}{2}]$  
and such that the gaps between rows and between columns is $\frac{1}{40n}$.

Inside each square $G$ from the family $\cG$ we take a $n \times m$-family $\cL_G$ of $\frac{2\kappa}{N} \times \frac{1}{11M}$-boxes;
it is possible to construct such a family since:
$$
\max\left\{ n\cdot \frac{2\kappa}{N} , \, m \cdot \frac{1}{11M} \right\}  = \max\left\{\frac{\kappa}{16n},\frac{1}{11n}\right\} < \frac{1}{10n} = 
\text{width of $G$.}
$$
Let $\cL \coloneqq \bigsqcup_{G \in \cG} \cL_G$; this is a family composed of $NM$ boxes.

Let $\cU$ be a $\frac{N}{2} \times M$-family of $\frac{2\kappa}{N} \times \frac{1}{11M}$-boxes
contained in the upper half-annulus $\cercle \times [\frac{1}{2}, 1]$.

\medskip
\noindent\textbf{Second step.}
We will need some auxiliary combinatorial data, namely certain coloring of our boxes.
We start by painting each $\cG$-square with a different color, and then we paint each $\cL_G$-box with the same color as $G$.
We claim that 
it is possible to paint each $\cU$-box with one of the $N$ previously chosen colors so that the following properties hold:
\begin{itemize}
	\item no row contains repeated colors (that is, exactly $N/2$ different colors appear in each row), and
	\item for any pair of distinct rows, there are at least $N/4$ colors that appear in one row but not in the other.
\end{itemize}
Indeed, if each choice of $N/2$ among $N$ colors can be identified with a function $f \colon \{1,\dots,N\} \to \{0,1\}$ such that $\sum_{k=1}^N f(k) = \frac{N}{2}$. The set $F$ of such functions was considered previously in the proof of \cref{t:lower_general}, where we have shown the existence of a set $F' \subset F$ which is $N/4$-separated w.r.t.\ the Hamming distance and has cardinality at least $(2N)^{-1/2}e^{\pi N/4^3}$: see estimate \eqref{e.Fprime}. Thus, by \eqref{e.def_M}, we can select $M$ distinct elements of the set $F'$. Each of these 
specifies a way of coloring a row of the family $\cU$; the order of the colors inside each row being arbitrary. This gives the desired coloring of the family $\cU$.

\medskip
\noindent\textbf{Third step.}
We will find a smooth conservative diffeomorphism  $h$ of the annulus that maps each $\cU$-box to a $\cL$-box of the same color by means of a translation, and which equals the identity near the boundary of the annulus. Essentially, this diffeomorphism exists because for each color $k$, there are \emph{at most} $M$ $\cU$-boxes of color $k$ (at most one box for each row), while there are \emph{exactly} $M=m\cdot n$ $\cL$-boxes of color $k$. Let us construct $h$ precisely.

We index the members of the family $\cU$ as $U_1$, $U_2$, \dots, $U_{NM/2}$ in such a way that $U_1$, \dots, $U_{N/2}$ form the bottom row, $U_{N/2+1}$, \dots, $U_{N}$ form the second from bottom row, and so on.
Then we select distinct $\cL$-boxes $L_1$, $L_2$, \dots, $L_{NM/2}$ in such a way that each $L_i$ has the same color as $U_i$, and whenever $L_i$ and $L_j$ have the same color and $i<j$ then $L_i$ is not above $L_j$.

For each $i=1,\dots,NM/2$, we will choose a smooth path $u_i \colon [0,1] \to \R^2$ starting from $u_i(0) = 0$
such that $t \in [0,1] \mapsto B_i(t) \coloneqq U_i + u_i(t)$ is a well-defined path of boxes in $\A$, starting at $B_i(0) = U_i$ and finishing at $U_i(1) = L_i$. We require the path of boxes $P_i \coloneqq \bigcup_{t\in [0,1]} B_i(t)$ to be disjoint from the set  
\begin{equation}\label{e.obstacle}
\partial \A \cup \bigcup_{j<i} L_j \cup \bigcup_{j>i} U_j \, .
\end{equation}
\begin{figure}[!h]
	\includegraphics[width=.55\textwidth]{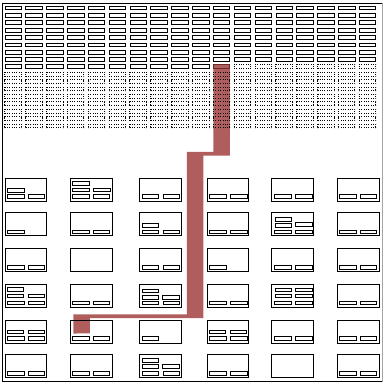}
	\caption{A path $P_i$ that avoids the obstacle set \eqref{e.obstacle}.}\label{f.tetris}
\end{figure}

These paths can be taken as follows: we start with the box $U_i$ and move it always either directly downwards or horizontally (like a Tetris piece). 
Note that $\frac{1}{40 n} > \frac{2\kappa}{N}$ (since $n \ge 3$); this means that the gaps between the squares of $\cG$ are greater than the width of the box.
Therefore it is possible to move between gaps and reach the destination $L_i$ avoiding the obstacle set \eqref{e.obstacle}: see \cref{f.tetris}.

Let $\phi_i \colon \A \to [0,1]$ be a smooth function that equals $1$ on the set $P_i$ and equals $0$ outside a small neighborhood of it (which is still disjoint from the set \eqref{e.obstacle}).
Now, writing $u_i(t) \eqqcolon (v_i(t), w_i(t))$, define a (non-autonomous) Hamiltonian $H_i \colon \A \times [0,1] \to \R$ by:
$$
H_i(\theta,\rho,t) \coloneqq \phi_i(\theta,\rho) \, \big[ w_i'(t) \rho - v_i'(t)\theta \big] \, .
$$
Let $f_i \in \Diff^\infty_\Leb(\A)$ be the time one map of the associated Hamiltonian flow.
Then $f_i$ translates the box $U_i$ to the box $L_i$, and equals the identity on the set \eqref{e.obstacle}.
It follows that the diffeomorphism
$$
h \coloneqq f_{NM/2} \circ \cdots \circ f_2 \circ f_1 
$$
translates each box $U_i$ to the corresponding $L_i$, and equals the identity on a neighborhood of $\partial \A$.

\medskip
\noindent\textbf{Fourth step.}
Let $\omega \colon [0,1]\to \R$ be any smooth function 
without critical points.
Consider the conservative flow $\Psi^t \coloneqq h \circ R^t_\omega \circ h^{-1}$ (see \cref{KAMacEmergence}).
We will estimate the emergence of the time $t$ maps from below, at an appropriate scale.

\begin{figure}[!h]
 \centering
 \includegraphics[width=.7\textwidth]{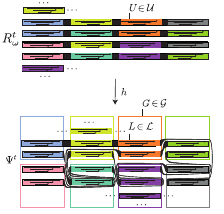}
 \caption{The flow $(\Psi^t)_t$.}
\label{KAMacEmergence}
\end{figure}

For each $\rho \in [0,1]$, let $\lambda_\rho$ denote Lebesgue measure on the circle $\cercle \times \{\rho\}$.
Recall from \eqref{e.erg_dec_Rt} that for every $t \neq 0$, the ergodic decomposition of Lebesgue measure on $\A$ with respect to $R^t_\omega$ is 
$\mathbf{e}^{R^t}_*(\Leb) = \int_0^1 \delta_{\lambda_\rho} \, d\rho$,
and in particular it is independent of $t$.
It follows that the ergodic decomposition of Lebesgue with respect to $\Psi^t \coloneqq h \circ R^t_\omega \circ h^{-1}$ is: 
\begin{equation}\label{e.erg_dec_Psi}
\hat\mu \coloneqq \mathbf{e}^{\Psi^t}_*(\Leb) = \int_0^1 \delta_{\tilde \lambda_\rho} \, d\rho 
\quad \text{where} \quad 
\tilde \lambda_\rho \coloneqq h_*(\lambda_\rho) \, .
\end{equation}
We need to estimate the quantization number of this measure.

Let $J$ be the set of $\rho\in[0,1]$ such that the circle $\cercle \times \{\rho\}$ intersects the boxes of the family $\cU$; then $J$ is a disjoint union of intervals $J_1$, \dots, $J_M$, each of them of length $\frac{1}{11M}$.

\begin{claim}\label{cl.together}
If $\rho$, $\rho' \in J$ belong to the same interval $J_i$ then:
$$
\mathsf{W}_1 (\tilde \lambda_\rho, \tilde \lambda_{\rho'}) \le \dfrac{1}{11 M} + 2\eta \, .
$$
\end{claim}

\begin{proof}[Proof of the claim]
Fix $\rho$, $\rho' \in J_i$.
We will use the following bound:
$$
\mathsf{W}_1 (\tilde \lambda_\rho, \tilde \lambda_{\rho'}) \le \int_{\cercle} \mathsf{d}\big(h(\theta,\rho), h(\theta,\rho')\big)  \, d\theta \, .
$$
Indeed, the right hand side is the cost of transporting each point $h(\theta,\rho)$ to $h(\theta,\rho')$.
Let $I \subset \cercle$ be the union of the projections of the $\cU$-boxes on the first coordinate; this is a union of
$\frac{N}{2}$ intervals of length $\frac{2\kappa}{N}$. 
Note that:
$$
\mathsf{d}\big(h(\theta,\rho), h(\theta,\rho')\big)  \le 
\begin{cases}
	1/(11 M) 		&\quad \text{if $\theta \in I$,}\\
	\diam \A \le 2	&\quad \text{otherwise;}
\end{cases}
$$
indeed if $\theta \in I$ then both points $(\theta,\rho)$ and $(\theta,\rho')$ belong to the same $\cU$-box $U$, which has height $\frac{1}{11M}$ and furthermore $h|_{U}$ is an isometry.
Finally, using the fact that $\Leb(I^\mathsf{c}) = 1-\kappa = \eta$, we obtain the asserted upper bound for the Wasserstein distance.
\end{proof}

\begin{claim}\label{cl.apart}
If $\rho \in J_i$, $\rho' \in J_j$ with $i \neq j$ then:
$$
\mathsf{W}_1 (\tilde \lambda_\rho, \tilde \lambda_{\rho'}) \ge \frac{1-3\eta}{80n} \, .
$$
\end{claim}

\begin{proof}[Proof of the claim]
Fix $\rho \in J_i$, $\rho' \in J_j$ with $i \neq j$.
Let $\cR'$ be the family of $\cU$-boxes that intersect the circle $\cercle \times \{\rho'\}$ (that is, a row of boxes).
Let $\cR$ be the family of $\cU$-boxes that intersect the circle $\cercle \times \{\rho\}$ and whose colors are distinct from those of the $\cR'$-boxes.
By construction, the family $\cR$ contains at least $N/4$ boxes; let $E$ be their union.
Since $\lambda_\rho(U) = 2\kappa/N$ for each $\cR$-box $U$, we have $\lambda_\rho(E) \ge \kappa/2 = (1-\eta)/2$.
So $F\coloneqq h(E)$ satisfies $\tilde\lambda_\rho(F) \ge (1-\eta)/2$.
The set $F$ is contained in the union of the $\cG$-squares whose colors appear in the family $\cR$.
Recall that $\frac{1}{40n}$ is the minimal separation between $\cG$-squares, so
if $V$ is the open $\frac{1}{40n}$-neighborhood of $F$ then $V$ does not intersect any $\cG$-square with other colors.
In particular, $\cR'$-boxes are disjoint from $h^{-1}(V)$.
Since the union of $\cR'$-boxes has $\lambda_{\rho'}$ measure equal to $\kappa = 1-\eta$, it follows that
$\tilde\lambda_{\rho'}(V) = \lambda_{\rho'}(h^{-1}(V)) \le \eta$.

Consider an arbitrary transport plan $\pi$ from $\tilde \lambda_\rho$ to $\tilde \lambda_{\rho'}$.
Then:
$$
\pi( F \times V^\mathsf{c} ) \ge \pi(F \times \A) - \pi(\A \times V) = \tilde \lambda_\rho (F) - \tilde\lambda_{\rho'}(V) \ge \frac{1-\eta}{2}-\eta = \frac{1-3\eta}{2} \, ,
$$
and so:
$$
\mathrm{cost}(\pi) 
= \int_{\A \times \A} \mathsf{d}(x,y) \, d\pi(x,y) 
\ge \int_{F \times V^\mathsf{c}} (\cdots) 
\ge \frac{1}{40n} \pi( F \times V^\mathsf{c} ) 
\ge \frac{1-3\eta}{80n} \, .
$$
Since this estimate holds for all transport plans $\pi$, we obtain the asserted lower bound for the Wasserstein distance.
\end{proof}

For every {$i \in \{1,\dots, M\}$}, let us fix a point $\rho_i$ in the interval $J_i$. The following measures 
{ $\hat{\mu}_1$, $\hat{\nu} \in \cM(\cM(\A))$}
correspond respectively to the ergodic decomposition of $\Psi^t|h(\mathbb T\times J)$ and to the probability measure equidistributed on the set $\{\tilde\lambda_{\rho_i} \st 1 \le i \le M\}$:
\begin{equation}\label{e.def_mu1}
\hat{\mu}_1 \coloneqq 11 \int_J  \delta_{\tilde \lambda_\rho} \, d\rho \in \cM(\cM(\A))
\quad\text{and}\quad
\hat\nu \coloneqq \frac{1}{M} \sum_{i=1}^M \delta_{\tilde\lambda_{\rho_i}} \, .
\end{equation}
It follows from \cref{cl.together} that $\mathsf{W}_1(\hat \nu, \hat \mu_1) \le \frac{1}{11 M} + 2\eta$; indeed each $\delta_{\tilde\lambda_{\rho}}$ with $\rho \in J_i$ can be transported to $\delta_{\tilde\lambda_{\rho_i}}$ { at} a cost no greater than $\frac{1}{11 M} + 2\eta$.

On the other hand, by \cref{cl.apart}, the measure $\hat{\nu}$ is equidistributed on a $\frac{1-3\eta}{80n}$-separated set 
of cardinality $M$.
So, by \cref{l:cost}, the $\mathsf{W}_1$-distance from $\hat\nu$ to any probability measure supported on $M/2$ points is bigger than $\frac{1-3\eta}{ 320 n}$. 
Therefore the $\mathsf{W}_1$-distance from $\hat\mu_1$ to any probability measure supported on $M/2$ points is bigger than:
\begin{equation}\label{e.def_epsilon} 
\frac{1-3\eta}{320n} - \left( \frac{1}{11 M} + 2\eta \right) \eqqcolon 11 \epsilon \, .
\end{equation}
In other words, 
$Q_{\hat\mu_1}(11 \epsilon) \ge \frac{M}{2}$.
By definitions \eqref{e.erg_dec_Psi}, \eqref{e.def_mu1}, we have $\hat\mu \ge \frac{1}{11} \hat\mu_1$,
and so \cref{l:submeasure} yields $Q_{\hat\mu}\left(\epsilon\right) \ge \frac{M}{2}$.
This quantization number is the metric emergence (by \cref{p:emergence_Q}), so we obtain:
\begin{equation}\label{e.conclusion}
\Eme_{\Leb}(\Psi^t) (\epsilon) \ge \frac{M}{2} \quad \text{for all } t\neq 0.
\end{equation}

\medskip
\noindent\textbf{Conclusion.}
If $\epsilon = \epsilon_n$ is defined by \eqref{e.def_epsilon} then $\epsilon_n \asymp  n^{-1} \asymp \epsilon_{n+1}$. 
For every sufficiently small number $\epsilon_*>0$, we can find $n \ge 3$ such that $\epsilon_{n+1} < \epsilon_* \le \epsilon_n$.
If $(\Phi^t)_t = (\Phi^t_n)_t$ is the flow constructed above, then for every $t \neq 0$ we have:
\begin{alignat*}{2}
\log \Eme_{\Leb}(\Psi^t) (\epsilon_*) \ge \log \Eme_{\Leb}(\Psi^t) (\epsilon_n) 
&\ge\log  	\frac{M}{2}		&\quad &\text{(by \eqref{e.conclusion})} \\
&\asymp n^2	&\quad &\text{(by \eqref{e.magnitude_M})}\\
&\asymp \epsilon_n^{-2}
\asymp \epsilon_*^{-2} \, .
\end{alignat*}
This ends the proof of \cref{p:main}.
\end{proof}

\begin{figure}[htb]
 \centering
\includegraphics[width=.75\textwidth]{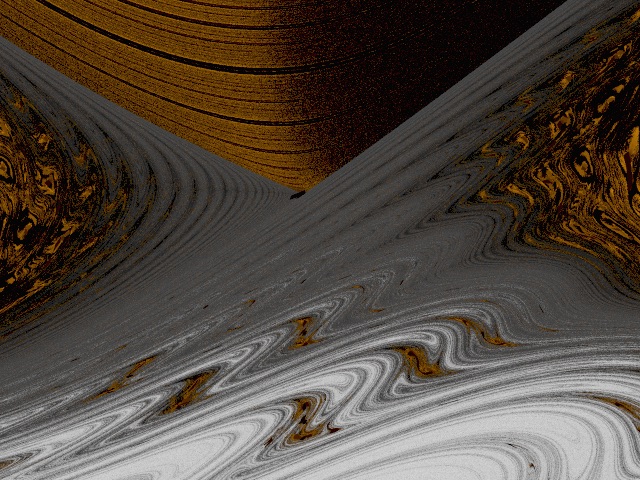}
 \caption{Interestingly, the geometry of the construction depicted at the bottom of \cref{KAMacEmergence} looks like the pattern obtained at the boundary of the so-called stochastic sea of the standard map for a certain parameter (depicted in gray in the above numerical experimentation).}
\end{figure}


\subsection{Construction of a smooth conservative flow with high emergence at every scale} \label{ss:main_example}

In this subsection, we will prove \cref{t:main}.
The main ingredient is \cref{p:main}.

\begin{proof}[Proof of \cref{t:main}]
Let us assume that the space $\cM(\A)$ is metrized with Wasserstein metric $\mathsf{W}_1$.
We will construct a conservative flow $(\Phi^t)_t$ on the annulus $\A$ whose metric emergence with respect to $\mathsf{W}_1$ is stretched exponential with exponent~$2$.
In the end of the proof we will see that the same holds 
if emergence is computed with respect to other Wasserstein metrics $\mathsf{W}_p$ or the L\'evy--Prokhorov metric $\mathsf{LP}$.

For each $i\ge 1$, let $h_i \in \Diff^\infty_\Leb(\A)$ be given by \cref{p:main} for $\epsilon_* = \epsilon_i  \coloneqq 2^{-i^2}$.
We define a smooth diffeomorphism between the annuli $\A$ and $\A_i \coloneqq \cercle\times [2^{-i}, 2^{-i+1}]$ as follows:
\[
g_i \colon (\theta, \rho)\in \A \mapsto  (\theta, 2^{-i}( \rho+1))\in \A_i \; .
\]
Let $h$ be the homeomorphism of the annulus $\A$ such that for each $i\ge 1$, 
\[
h(\A_i) = \A_i \qand
h | \A_i \coloneqq g_i \circ h_i \circ g_i^{-1} \; .
\]
Since each $g_i$ has constant jacobian, $h$ is conservative. 
Furthermore, $h$ equals the identity on the boundary of the annulus and is smooth on its interior.
Next, fix a smooth function $r \colon \R \to \R$ that vanishes on $(-\infty,0] \cup [1,+\infty)$ and is positive on the interval $(0,1)$. Let $(\eta_i)$ be a sequence of positive numbers converging very rapidly to $0$. 
Define a function $\zeta \colon [0,1] \to \R$ by:
\[
\zeta(\rho) \coloneqq 
\begin{cases}
	\eta_i \cdot r(2^i\rho-1)	&\quad\text{if $\rho \in [2^{-i}, 2^{-i+1}]$;} \\
	0					&\quad\text{if $\rho=0$.}
\end{cases}
\]
If $\eta_i$ tends to $0$ sufficiently rapidly then $\zeta$ becomes a smooth function.
Furthermore, it is positive Lebesgue a.e. 
Let $\omega \colon [0,1] \to \R$ be the smooth function such that $\omega(0) = \omega'(0) = 0$ and $\omega'' = \zeta$.
Observe that $\omega'$ is strictly positive on $(0,1]$. 
Furthermore, the $C^i$-norm of $\omega|[2^{-i} , 2^{-i+1}]$ is small when $(\eta_j)_{j\ge i}$ is small.
Hence we can choose inductively $\eta_i$ sufficiently small so that the push forward of the vector field $\partial_t R^t_\omega$ by $h$, namely
\[
h_*\partial_t R^t_\omega :(\theta, \rho)\mapsto 
\big[\omega\circ p_2\circ h^{-1}(\theta, \rho)\big]
\cdot \partial_\theta h\circ h^{-1}(\theta, \rho) \, ,
\]
has $C^i$-norm restricted to each $\A_i$ smaller than $1$ for every $i$.

Thus this vector field and its flow  
$\Phi^t = h \circ R^t_\omega \circ h^{-1}$ are smooth. 
The construction of the smooth conservative flow $(\Phi^t)_t$ is completed, and is depicted in \cref{f.onion}.
We are left to show that the flow has stretched exponential emergence with exponent $2$. 
Of course, the homeomorphism $h$ cannot be smooth on the whole annulus, because otherwise  the flow would have polynomial emergence (by \cref{emergence poly} and lemmas from \cref{ss:quant_properties}).

\begin{figure}[hbt]
	\includegraphics[width=.8\textwidth]{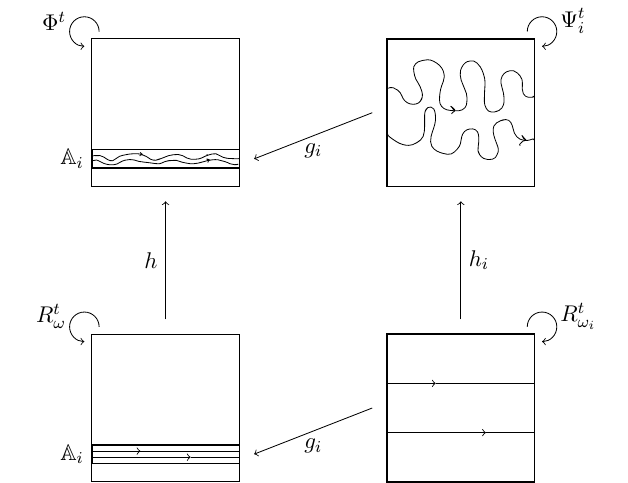}
	\caption{Proof of \cref{t:main}.}\label{f.onion}
\end{figure}

Let $\hat\mu \in \cM(\cM(\A))$ denote the ergodic decomposition of $\Leb$ with respect to $\Phi^t$
(which is indeed independent of $t \neq 0$, since it is given by formula \eqref{e.erg_dec_Psi}).
Recalling \cref{p:emergence_Q}, we have $\Eme_\Leb(\Phi^t)(\epsilon) = Q_{\hat{\mu}}(\epsilon)$ for every $\epsilon$.

\begin{claim}\label{cl.dejavu} 
For every $i\ge 1$ and $t\neq 0$ we have 
$\Eme_\Leb(\Phi^t)(\epsilon_{i+1}) \ge \exp(C \epsilon_i^{-2})$,
where $C>0$ is a constant.
\end{claim}

\begin{proof}[Proof of the claim]
Let $\mu_i \coloneqq 2^i \, \Leb|\A_i = g_{i*}(\Leb)$; this is a $\Phi^t$-invariant probability measure.
Its ergodic decomposition $\hat{\mu}_i$ is bounded from above by $2^i \hat{\mu}$ and so, by \cref{l:submeasure}, for all $\epsilon>0$ we have:
$$
Q_{\hat\mu}(\epsilon) \ge Q_{\hat\mu_i}(2^i \epsilon) .
$$

Let $\omega_i \coloneqq \omega\circ g_i$, and 
consider the conservative flow $\Psi_i^t \coloneqq h_i \circ R^t_{\omega_i} \circ h_i^{-1}$ and its ergodic decomposition $\hat\nu_i \coloneqq \mathbf{e}_*^{\Psi_i^t}(\Leb)$ (for $t \neq 0$).
Note that there exists a $1$-Lipschitz retraction $p_i \colon \A \to \A_i$.
Let $q_i \colon \A \to \A$ be the map $q_i \coloneqq g_i^{-1} \circ p_i$. 
By \cref{l:push_erg} we have $\hat\nu_i = q_{i**}(\hat\mu_i)$, that is, $\hat\nu_i$ is the push-forward of $\hat\mu_i$ under the map $q_{i*} \colon \cM(\A) \to \cM(\A)$.
Since $q_i$ is $2^i$-Lipschitz, so is $q_{i*}$.
Hence, by \cref{l:quant_Lip},
$$
Q_{\hat\mu_i}(2^i \epsilon) \ge Q_{\hat\nu_i}(2^{2i} \epsilon) \, .
$$
In summary, we have shown that $Q_{\hat\mu}(\epsilon) \ge Q_{\hat\nu_i}(2^{2i} \epsilon)$, that is,
$$
\Eme_\Leb(\Phi^t)(\epsilon) \ge \Eme_{\Leb}(\Psi^t_i)(2^{2i} \epsilon) \, .
$$
Taking $\epsilon = 2 \epsilon_{i+1}$ and noting that $2^{2i} \epsilon =  \epsilon_i$, we obtain:
$$
\Eme_\Leb(\Phi^t)(\epsilon_{i+1}) \ge
\Eme_\Leb(\Phi^t)(2 \epsilon_{i+1}) \ge \Eme_{\Leb}(\Psi^t_i)(\epsilon_i) \ge 
\exp(C \epsilon_i^{-2}) \, ,
$$
where the last inequality is the main property of the flow $(\Psi_i^t)_t$, coming from \cref{p:main}. 
This proves the \lcnamecref{cl.dejavu}.
\end{proof}

Next, we claim that: 
\begin{equation}\label{e.dejavu}
\liminf_{\epsilon\to 0}
\frac{\log \log Q_{\hat\mu}(\epsilon)}{-\log \epsilon} 
\ge 2 \, .
\end{equation}
Indeed, given a small $\epsilon>0$, let $i$ be such that $\epsilon \in [\epsilon_{i+2}, \epsilon_{i+1}]$.
We have $Q_{\hat\mu}(\epsilon)\le Q_{\hat\mu}(\epsilon_{i+1})$ 
and so, using \cref{cl.dejavu}, 
\[
\frac{\log \log Q_{\hat\mu}(\epsilon)}{-\log \epsilon} \ge
\frac{\log \log Q_{\hat\mu}(\epsilon_{i+1})}{-\log \epsilon_{i+2}} \ge
\frac{\log C - 2 \log \epsilon_i}{-\log \epsilon_{i+2}} 
\, .
\]
The right-hand side tends to $2$ as $i \to \infty$, so \eqref{e.dejavu} follows.

Inequality \eqref{e.dejavu} means that the lower quantization order of $\hat\mu$ is at least~$2$, that is, $\underline{\qo}(\hat\mu)\ge 2$. 
Up to this moment we were assuming that $\cM(\A)$ is metrized with Wasserstein metric $\mathsf{W}_1$, but now let us use 
any Wasserstein metric $\mathsf{W}_p$, $1\le p < \infty$, or the L\'evy--Prokhorov metric $\mathsf{LP}$.
By inequalities \eqref{Wp_monotone}, \eqref{compaWpLP} from p.~\pageref{compaWpLP}, we still have $\underline{\qo}(\hat\mu)\ge 2$ with respect to the other metrics.
On the other hand, by \cref{l:quant_cover,t:mo_box,r.KT}, we have:
\[
\overline{\qo}(\hat{\mu}) \le \overline{\mo}(\cM(\A)) \le \overline{\dim}(\A) = 2 \, .
\]
Therefore $\qo(\hat{\mu})=2$.
This means exactly that:
\[
\lim_{\epsilon\to 0} \frac{\log \log \Eme_{\Leb}(\Phi^t) (\epsilon)}{-\log \epsilon} = 2 \, ,
\]
which completes the proof of \cref{t:main}.
\end{proof}

\section{Genericity of high emergence} \label{s:generic} 

In this \lcnamecref{s:generic}, we will prove our \cref{coroC,coroD} on the genericity of high emergence among surface diffeomorphisms. Both proofs are based on \cref{p:main}.  Another fundamental tool is the creation of \emph{periodic spots}; we recall the relevant results in \cref{ss:spots}. Furthermore, for the conservative \cref{coroC}, we also need a KAM theorem, which is discussed along with some of its consequences in \cref{ss:KAM}.

\emph{Throughout this section, let $(M,\Leb)$ be a compact surface endowed with a normalized smooth volume (i.e.\ area) measure.}

\subsection{Creation of periodic spots}\label{ss:spots}

\cref{coroC,coroD} are proved using the following concept:

\begin{definition}[Periodic spot] 
An open subset $O\subset M$ is a \emph{periodic spot} for a continuous self-map $f$ of $M$ if  there exists $p\ge 1$ such that $f^p(O)=O$ and the restriction $f^p|O$ is the identity map on $O$.
\end{definition}

Diffeomorphisms displaying a periodic spot appears densely in many open subsets of dynamical systems.
In the conservative setting we have:

\begin{otherthm}[{\cite[Thrm.~5]{GTS07}} and {\cite[Thrm.~1]{GT10}}]
\label{dense_periodic_spotC}
For every $r\in [1, \infty]$, if $\cU$ is the open subset of $\Diff^r_\Leb(M)$ formed by dynamics having an elliptic periodic point, then there exists a dense subset $\cD\subset \cU$ formed by dynamics displaying a periodic spot. 
\end{otherthm} 

The statement above follows from the combination of \cite[Thrm.~5]{GTS07} and \cite[Thrm.~1]{GT10}.

In the dissipative setting we have:

\begin{otherthm}[Turaev, {\cite[Lemma~2]{Turaev}}]\label{dense_periodic_spotD}
For every $r\in [2, \infty]$, there exists a non-empty open set $\cU\subset \Diff^r(M)$ and a dense set $\cD\subset \cU$ formed by dynamics displaying a periodic spot. 
\end{otherthm} 

We add that the set $\cU$ in \cref{dense_periodic_spotD} contains the absolute Newhouse domain: it is formed by diffeomorphisms displaying a horseshoe having a robust homoclinic tangency, a volume expanding periodic point, and a volume contracting periodic point.

\subsection{KAM and stability of high emergence}\label{ss:KAM}

A \emph{twist map} is a conservative diffeomorphism $f_0 \colon \A \to \A$ of the form $f_0 = R^t_\omega$, where
$\omega \colon [0,1]\to \R$ is a smooth function without critical points, and $t \neq 0$. 

\begin{otherthm}[{Moser--P\"oschel's twist mapping theorem \cite{Moser}, \cite{Poschel}, \cite[\S~3.2.1]{BroerS}}] \label{t:KAM}
Let $f$ be a twist map.
Fix a number $\eta>0$ a neighborhood $\cU$ of the identity map in $\Diff^\infty(\A)$. 
Then there exist a closed subset $D \subset (0,1)$ with Lebesgue measure at least $1-\eta$
and a neighborhood $\cV$ of $f$ in $\Diff^\infty_\Leb(\A)$ such that for every $g \in \cV$, 
there exists $h\in \cU$ such that
the map $f|\cercle\times D$ is conjugate to $g|h(\cercle\times D)$ via $h$:
$$
g \circ h(z) = h \circ f(z),\quad \forall z\in \cercle\times D\; .
$$
\end{otherthm}

As a corollary of \cref{t:KAM}, we will prove below:

\begin{corollary}\label{c:robust_erg_dec}
Let $f \in \Diff^\infty_\Leb(M)$ be a conservative surface diffeomorphism that acts as a twist map on a embedded annulus $A \subset M$; more precisely, assume that there exist a smooth embedding $h_1 \colon \A \to M$ with constant jacobian and image $h_1(\A) = A$ and a twist map $f_0 \colon \A \to \A$
such that $f \circ h_1 = h_1 \circ f_0$. 
Then for every $\epsilon_1 > 0$, for every $g \in \Diff^\infty_\Leb(M)$ sufficiently close to $f$, there exists a $g$-invariant embedded sub-annulus
$B \subset A$ 
such that $\Leb(A \setminus B) < \epsilon_1$ and  
$$
\mathsf{W}_1 \left(\mathbf{e}^g_*(\mu_B), \mathbf{e}^f_*(\mu_A)  \right) < \epsilon_1 \, ,
$$
where $\mu_A$ and $\mu_B$ are the normalized Lebesgue measures on $A$ and $B$, respectively.
\end{corollary}

In the case that $f$ itself is a twist map (so $M = A = \A$, $f=f_0$, and $h_1 = \id$), we can actually take $B = A$.
In particular, $f$ becomes a continuity point for the ergodic decomposition of Lebesgue measure (c.f.\ \cref{r:semicont_erg_dec}).

We will use the following general estimate:

\begin{lemma}\label{l.L1_W1}
Let $(X,\mathsf{d})$ be a compact metric space and $\mu$, $\nu \in \cM(X)$.
If $\nu$ is absolutely continuous w.r.t.\ $\mu$, with density $r \coloneqq \frac{d\nu}{d\mu}$, then:
$$
\mathsf{W}_1(\nu,\mu) \le (\diam X) \|r-1\|_{L^1(\mu)} \, . 
$$
\end{lemma}

\begin{proof}
This is an immediate consequence of the Kantorovich--Rubinstein duality formula \cite[p.~207]{Villani}.
\end{proof}

\begin{proof}[Proof of \cref{c:robust_erg_dec}]	
Let us first consider the simpler case where $f = f_0$ is a twist map on $M = A = \A$, and so $h_1 = \id$.
Let $\eta>0$ be small, and let $\cU$ be a small neighborhood of the identity map in $\Diff^\infty(\A)$.
We apply \cref{t:KAM}, obtaining a set $D\subset[0,1]$ and a neighborhood $\cV$ of the twist map $f$ in $\Diff^\infty_\Leb(\A)$.
Take an arbitrary $g \in \cV$.
We need to prove that the ergodic decompositions of Lebesgue measure with respect to $f$ and $g$ are approximately the same.

Consider the (non-conservative) diffeomorphism $\tilde f \coloneqq h^{-1} \circ g \circ h$; by construction it equals $f$ on $\cercle \times D$.
The measure $\nu \coloneqq h_*^{-1}(\Leb)$ is $\tilde f$-invariant.
Let $c \coloneqq \nu (\cercle \times D)$ (a number close to $1$),
and let $\nu_1 \coloneqq c^{-1} \cdot \nu|\cercle \times D$.
Then $\nu_1$ is $\tilde f$-invariant.
Since $\tilde f$ equals $f$ on $\supp \nu_1 = \cercle \times D$, this measure is also $f$-invariant.
We will prove that the four ergodic decompositions below are close to each other:
$$
\circled{1} \coloneqq \mathbf{e}_*^{g}(\Leb), \quad 
\circled{2} \coloneqq \mathbf{e}_*^{\tilde f}(\nu), \quad 
\circled{3} \coloneqq \mathbf{e}_*^{\tilde f}(\nu_1) = \mathbf{e}_*^{f}(\nu_1), \quad 
\circled{4} \coloneqq \mathbf{e}_*^{f}(\Leb).
$$

Since $h$ is close to identity, the distances $\mathsf{d}(h(x),x)$ are uniformly bounded by a small constant $\epsilon_2$. Therefore:
$$
\forall \xi \in \cM(\A) \, , \quad
\mathsf{W}_1 ( h_* \xi, \xi) \le \int \mathsf{d}(h(x),x) \, d\xi(x) \le \epsilon_2 \, ,
$$
That is, $h_* \colon (\cM(\A),\mathsf{W}_1) \to (\cM(\A),\mathsf{W}_1)$ is $\epsilon_2$-close to the identity map. 
Repeating the argument, we see that $h_{**}$ is also close to the identity. By \cref{l:push_erg}, $h_{**}(\circled{2})=\circled{1}$; this proves that the measures $\circled{1}$ and $\circled{2}$ are close.

The Radon--Nikodym derivative $r \coloneqq \frac{d \nu}{d \Leb}$ is smooth and uniformly close to $1$.
We have:
$$
\frac{d \nu_1}{ d\nu 
} = c^{-1} \, \mathbf{1}_{\cercle \times D} \, .
$$
(Here $\mathbf{1}$ denotes characteristic function.)
On the other hand, for all $(\theta,\rho) \in \A$, the empirical measure $\mathbf{e}^f(\theta,\rho)$ is Lebesgue on the circle $\cercle \times \{\rho\}$, denoted $\lambda_\rho$.
It follows that the ergodic decomposition of $\Leb$ and $\nu_1$ with respect to $f$ are:
$$
\circled{4} = \mathbf{e}_*^{f}(\Leb) = \int_{0}^1 \delta_{\lambda_\rho} \, d\rho \, ,
$$
and
$$
\circled{3} = \mathbf{e}_*^{f}(\nu_1) = 
c^{-1} \int_D \bar{r}(\rho) \, \delta_{\lambda_\rho} \, d\rho \, ,
$$
where $\bar{r}(\rho) \coloneqq \int r \, d\lambda_\rho$.
So $\circled{3}$ is absolutely continuous with respect to $\circled{4}$, with density:
$$
\frac{d\circled{3}}{d\circled{4}} \big(\delta_{\lambda_\rho}\big) = c^{-1} \bar{r}(\rho) \mathbf{1}_{D}(\rho) \, . 
$$
This function is close to $1$ in $L^1(\circled{4})$; so \cref{l.L1_W1} implies that the measures $\circled{3}$ and $\circled{4}$ are close.
Finally, we have:
$$
\frac{d\circled{3}}{d\circled{2}}  = c^{-1} \, \mathbf{1}_S  \, , 
\quad \text{where }
S \coloneqq \supp \circled{3} = \{ \delta_{\lambda_\rho} \st \rho \in D\} \, .
$$
This function is close to $1$ in $L^1(\circled{2})$,  
so \cref{l.L1_W1} implies that the measures $\circled{2}$ and $\circled{3}$ are close.
The upshot is that $\circled{1}=\mathbf{e}_*^{g}(\Leb)$ and $\circled{4}=\mathbf{e}_*^{f}(\Leb)$ are close.
This completes the proof of the \lcnamecref{c:robust_erg_dec} in the case $f$ is a twist map.

The general situation can be reduced to the previous case. 
Indeed, \cref{t:KAM} also ensures that if $f \colon M \to M$ acts as a twist map on a annulus $A$, then any perturbation $g$ of $f$ admits a $g$-invariant sub-annulus $B \subset A$ which is close to $A$. { Then the proof is verbatim the same by substituting the measures $\lambda_\rho$ by their pushforward by $h_1$, $g$ by $h_1\circ g\circ h_1^{-1}$, and $h$ by $h_1\circ h\circ h_1^{-1}$.}  
\end{proof}

As another corollary of \cref{t:KAM}, we obtain open sets with at least polynomial emergence, so justifying an assertion made in \cref{ss:horizontal}. (Readers anxious to see the proof of \cref{coroC} may skip this.)

\begin{corollary}\label{c:robust_poly}
Under the same hypotheses as \cref{c:robust_erg_dec},
there exists $C>0$ such that for every $g \in \Diff^\infty_\Leb(M)$ sufficiently close to $f$, its emergence of $g$ with respect to the $\mathsf{W}_1$ metric satisfies:
$$
\Eme_\Leb(g)(\epsilon) \ge C \epsilon^{-1} \, , \quad  \forall \epsilon>0 \, ,
$$
\end{corollary}

Note that this is \emph{not} a consequence of \cref{c:robust_erg_dec} by itself, since we bound the emergence of the perturbations at \emph{every} scale.

\begin{proof} 
We will provide a proof in the case that $f$ itself is a twist map (so $M = A = B = \A$, $f=f_0$, and $h_1 = \id$), leaving for the reader to adapt the proof for the general situation.

Let $g \in \Diff^\infty_\Leb(\A)$ be a perturbation of $f$.
Applying \cref{t:KAM}, we obtain $h \in \Diff^\infty(\A)$ close to identity such that $g \circ h = h \circ f$ on $\cercle \times D$, where $D \subset (0,1)$ is a closed set with almost full measure; say at least $1/2$.
We can assume that $h^{\pm 1}$ are $2$-Lipschitz and have jacobian at most $2$.
As in the proof of \cref{c:robust_erg_dec}, let 
$\tilde f \coloneqq  h^{-1} \circ g \circ h 
$ and $\nu \coloneqq h_*^{-1}(\Leb)$;
then $\nu$ is { $\tilde f$-invariant.}

Consider the ergodic decompositions $\hat\mu \coloneqq \mathbf{e}_*^g(\Leb)$ and $\hat\nu \coloneqq \mathbf{e}_*^{\tilde{f}}(\nu)$.
Then, for arbitrary $\epsilon>0$, 
\[
\begin{aligned}
\Eme_{\Leb}
(g)(\epsilon) 
&=   Q_{\hat\mu}(\epsilon)			 							&\quad &\text{(by \cref{p:emergence_Q})} \\ 
&\ge Q_{h^{-1}_{**}(\hat\mu)}(\mathrm{Lip}(h^{-1}_*) \epsilon)	&\quad &\text{(by \cref{l:quant_Lip})} \\ 
&\ge Q_{\hat\nu}(2 \epsilon)									&\quad &\text{(by \cref{l:push_erg,l:push_Lip})} \, .
\end{aligned}
\]
Note that $\nu \ge \tfrac{1}{2} \Leb$, by the bound on the jacobian.
Let $\mu_1$ be the normalized Lebesgue measure on $\cercle \times D$.
Since $\Leb(D) \ge \tfrac{1}{2}$ we have $\mu_1 \le 2 \Leb$, and so $\nu \ge \tfrac{1}{4} \mu_1$.
Furthermore, $\mu_1$ is also $\tilde{f}$-invariant so the ergodic decomposition $\hat{\mu}_1 \coloneqq \mathbf{e}_*^{\tilde{f}}(\nu)$ is well-defined.
We have $\hat\nu \ge \tfrac{1}{4} \hat\mu_1$ and so, by \cref{l:submeasure},
\[
{\Eme_{ \Leb}(\epsilon) }\ge Q_{\hat\nu}(2 \epsilon) \ge  Q_{\hat\mu_1}(8 \epsilon) \, . 
\]
Since $\tilde{f}$ equals $f$ on the support of $\mu_1$, 
we have:
$$
\hat\mu_1 = 
\frac{1}{\Leb(D)} \int_D \delta_{\lambda_\rho} \, d\rho \, ,
$$
where $\lambda_\rho$ denotes Lebesgue measure on the circle $\cercle \times \{\rho\}$.
Similarly to the proof of \cref{emergence poly}, the measure $\hat\mu_1$ is supported on a set which is isometric to $D$ under the isometry $\lambda_\rho \mapsto \rho$; moreover, the isometry carries $\hat\mu_1$ to the normalized Lebesgue measure on $D$ (call it $\lambda$).
Therefore  $Q_{\hat\mu_1}(8 \epsilon) = Q_{\lambda}(8 \epsilon)$.

We are left to estimate the quantization number of the measure $\lambda$.
Consider its distribution function $F \colon [0,1] \to [0,1]$ defined by $F(x) \coloneqq \lambda([0,x])$.
Since $\lambda \le 2 \Leb$, the function $F$ is $2$-Lipschitz.
Furthermore, $F_*(\lambda) = \Leb$.
So, by \cref{l:quant_Lip},
$$
Q_{\lambda}(8 \epsilon) \ge Q_{\Leb}(16 \epsilon) \, .
$$
We have seen in \cref{quantization 1d} that quantization number of $1$-dimensional Lebesgue measure is $Q_{\Leb}(\epsilon) \asymp \epsilon^{-1}$.
We conclude that $\Eme_\mu(g)(\epsilon)$ is at least of the order of $\epsilon^{-1}$, as we wanted to show.
\end{proof}

\subsection{Genericity of high emergence: conservative setting}\label{s:generic_cons} 

Here is a consequence of \cref{p:main}, combined with \cref{c:robust_erg_dec}:

\begin{lemma}\label{forC}
Suppose that $f\in \Diff^\infty_\Leb(M)$ admits a periodic spot~$O$. 
Let $\epsilon_0 > 0$, and let $\cU \subset \Diff^\infty_\Leb(M)$ be a neighborhood of $f$. 
Then there exists a nonempty open set $\cV \subset \cU$ such that 
for every $g\in \cV$, its metric emergence  w.r.t.\ $\mathsf{W}_1$ metric satisfies:
\[\sup_{\epsilon<\epsilon_0} \frac{\log \log \Eme_\Leb(g)(\epsilon)}{\log \epsilon} \ge 2-\epsilon_0\; .\]
\end{lemma}

\begin{proof}
Assume that $f$ has a periodic spot $O$.
For simplicity of writing, let us assume that $O$ consists of fixed points.

Let $\hat{\A} \coloneqq \cercle \times [-1,2]$.
Take a smooth embedding $h_1 \colon \hat{\A} \to O$ with constant Jacobian $J$.
Fix a small $\epsilon_*>0$; how small it needs to be will become apparent at the end.

Let $\omega \colon \R \to \R$ be a smooth function that has no critical points in $[0,1]$ and vanishes outside $[-1,2]$.
By \cref{p:main}, we can find $h \in \Diff_\Leb^\infty(\A)$ that equals the identity on the neighborhood of the boundary
such that the maps $\Psi^t \coloneqq h \circ R^t_\omega\circ h^{-1}$ has high emergence at scale $\epsilon_*$:
$$
\Eme(\Psi^t)(\epsilon_*) \ge \exp(C\epsilon_*^{-2}) \, , \quad \forall t \neq 0 \, ,
$$
where $C$ is a constant.
We fix $t \neq 0$ very close to $0$ and write $\psi \coloneqq \Psi^t$.
We can extend $h$ and $\psi$ to smooth conservative diffeomorphisms $\hat h$ and $\hat\psi$ of the bigger annulus $\hat\A$, putting $\hat{h}(\theta,\rho) \coloneqq (\theta,\rho)$ and $\hat\psi(\theta,\rho) \coloneqq (\theta+t\omega(\rho),\rho)$ for $(\theta,\rho) \in \hat{\A}\setminus \A$. 
Define $\tilde f \colon M \to M$ by:
$$
\tilde{f}(x) \coloneqq 
\begin{cases}
	h_1 \circ \hat\psi \circ h_1^{-1}(x) &\quad\text{if $x \in h_1(\hat{\A})$;} \\
	f(x) &\quad\text{otherwise.}
\end{cases}
$$
Then $\tilde f$ is a smooth conservative diffeomorphism, and it is $C^\infty$-close to $f$ (since $t$ is close to $0$).
So we can assume that $\tilde f$ belongs to the given neighborhood $\cU$ of $f$.
Note that $\tilde f$ acts as a twist map on the embedded annulus $A \coloneqq h_1(\A)$, which has measure $J$ (the jacobian of $h_1$). 

Let $g \in \Diff^\infty_\Leb(M)$ be a small perturbation of $\tilde f$.
By \cref{c:robust_erg_dec}, $g$ admits an invariant sub-annulus $B \subset A$
such that:
$$
\Leb(A \setminus B) < \tfrac{J}{2} \quad  \text{and} \quad
\mathsf{W}_1 \left(\mathbf{e}^g_*(\mu_B), \mathbf{e}^{\tilde f}_*(\mu_A)  \right) < L^{-1}\epsilon_* \, ,
$$
where $\mu_A$ (resp.\ $\mu_B$) is the normalized Lebesgue measure on $A$ (resp.\ $B$).
Since $\mu_A  = h_{1*}(\Leb_\A)$, by \cref{l:push_erg}, $\mathbf{e}^{\tilde f}_*(\mu_A) = h_{1**}(\mathbf{e}^{\psi}_*(\Leb_\A))$.
Let $L$ be the Lipschitz constant of $h_1$;
then, by \cref{l:push_Lip,l:quant_Lip},
$$
Q_{\mathbf{e}^{\tilde f}_*(\mu_A)} ( L^{-1}\epsilon_* ) \ge Q_{\mathbf{e}^{\psi}_*(\Leb_\A)} (\epsilon_*) \ge \exp(C\epsilon_*^{-2}) \, .
$$
It follows from \cref{l:continuity} that:
$$
Q_{\mathbf{e}^{g}_*(\mu_B)} ( 2 L^{-1}\epsilon_* ) \ge Q_{\mathbf{e}^{\tilde f}_*(\mu_A)} ( L^{-1} \epsilon_* ) \ge \exp(C\epsilon_*^{-2}) \, .
$$
Since $\Leb(B) \ge \frac{J}{2}$, we have $\mu_B \le 2J^{-1} \Leb$ and so $\mathbf{e}^{g}_*(\Leb) \le 2J^{-1}  \mathbf{e}^{g}_*(\mu_B)$.
It follows from \cref{l:submeasure} that:
$$
Q_{\mathbf{e}^{g}_*(\Leb)} (  4L^{-1}J^{-1} \epsilon_* )  \ge Q_{\mathbf{e}^{g}_*(\mu_B)} ( 2 L^{-1}\epsilon_* )  \ge \exp(C\epsilon_*^{-2}) \, .
$$
Let $\epsilon \coloneqq 4L^{-1}J^{-1}\epsilon_*$.
Since $\epsilon_*$ is very small, we conclude that $\epsilon < \epsilon_0$ and 
$$
\frac{\log \log \Eme_\Leb(g)(\epsilon)}{\log \epsilon} \ge 2-\epsilon_0 \, .
$$
Therefore the neighborhood $\cV$ of $\tilde f$ formed by the perturbations $g$ has the required properties.
\end{proof}

\begin{proof}[Proof of \cref{coroC}] 
Consider the following two subsets of $\Diff^\infty_\Leb(M)$:
\begin{itemize}
	\item $\cW$ is the set of weakly stable diffeomorphisms, i.e., those that robustly have only hyperbolic periodic points (if any);
	\item $\cU$ is the set of diffeomorphisms that admit at least one elliptic periodic point.
\end{itemize}
These two sets are open and disjoint.
Furthermore, since the periodic points of a generic area preserving map are either hyperbolic or elliptic, 
the union $\cW \cup \cU$ is dense in $\Diff^\infty_\Leb(M)$.

By \cref{dense_periodic_spotC}, there is a dense subset $\cD\subset \cU$ formed by diffeomorphisms displaying a periodic spot.
For each $f\in \cD$, let $(\cU_{f,n})$ be a neighborhood basis for $f$.
By \cref{forC}, there exists a nonempty open subset $\cV_{f,n} \subset \cU_{f,n}$ such that:
\[
\forall g \in \cV_{f,n} \, , \quad
\sup_{\epsilon<1/n} \frac{\log \log \Eme_\Leb(g)(\epsilon)}{\log \epsilon} \ge 2-\frac{1}{n}\; .
\]
The set $\cO_n \coloneqq \bigcup_{f\in \cD} \cV_{f, n}$ is open and dense in $\cU$.
Then $\cR\coloneqq \bigcap_n \cO_n$ is a  residual subset of $\cU$  which satisfies:
\[
\forall g \in \cR \, , \quad 
\limsup_{\epsilon \to 0 } \frac{\log \log \Eme_\Leb(g)(\epsilon)}{\log \epsilon} \ge 2 \, .
\]
Thus $\cW \cup \cR$ is a residual subset of $\Diff^\infty_\Leb(M)$ formed by diffeomorphisms that are either weakly stable or have $\limsup$ stretched exponential emergence with exponent $2$.
\end{proof}
	
\subsection{Genericity of high emergence: dissipative setting} \label{s:generic_dissip}

Here is another consequence of \cref{p:main}:

\begin{lemma}\label{forD}
Let $r \in [1,\infty]$. 
Suppose that $f\in \Diff^r(M)$ admits a periodic spot~$O$. 
Let $\epsilon_0 > 0$, and let $\cU \subset \Diff^\infty(M)$ be a neighborhood of $f$. 
Then there exists a nonempty open set $\cV \subset \cU$ such that 
for every $g\in \cV$, its metric emergence  w.r.t.\ $\mathsf{W}_1$ metric satisfies:
\[\sup_{\epsilon<\epsilon_0} \frac{\log \log \Eme_\Leb(g)(\epsilon)}{\log \epsilon} \ge 2-\epsilon_0\; .\]
\end{lemma}

\begin{proof}
Let us first consider the simpler case where $M$ is the annulus and $f$ is the identity map. 
Let $\epsilon_0>0$ be given.
Fix a positive $\epsilon_* < \epsilon_0$ small enough such that:
$$
\frac{\log C - 2 \log\epsilon_*}{\log 4 - \log \epsilon_*} \ge 2 - \epsilon_0 \, ,
$$
where $C>0$ is the constant from \cref{p:main}.
Choose and fix a smooth function $\omega \colon [0,1] \to \R$ without critical points.
Applying \cref{p:main}, 
we obtain a smooth conservative diffeomorphism $h \colon \A \to \A$ that equals identity on a neighborhood of the boundary of the annulus, such that the flow $\Psi^t \coloneqq h \circ R^t_\omega \circ h^{-1}$
has the following property:
\[
\forall t \neq 0, \quad 
\Eme_{\Leb}(\Psi^t) (\epsilon_*) \ge \exp(C \epsilon_*^{-2}) \, .
\]

For each $\rho \in [0,1]$, let $\lambda_\rho$ denote Lebesgue measure on the circle $\cercle \times \{\rho\}$, and let $\tilde\lambda_\rho \coloneqq h_*(\lambda_\rho)$ be its push-forward under $h$.
So $\tilde\lambda_\rho$ is supported on the curve $\cC_\rho \coloneqq h(\cercle \times \{\rho\})$.
Consider the following sequence of elements of $\cM(\cM(\A))$:
$$
\hat\mu_n \coloneqq 
\frac{1}{n}\sum_{i=0}^{n-1} \delta_{\tilde \lambda_{(i+.5)/n}} \, . 
$$
Note that the sequence $(\hat\mu_n)$ tends to the measure $\hat \mu$ defined by \eqref{e.erg_dec_Psi},
which is exactly the ergodic decomposition of any $\Psi^t$ ($t\neq 0$).
By \cref{l:continuity}, if $n$ is large enough 
then
$Q_{\hat{\mu}_n}(\epsilon_*/2)$ is at least $Q_{\hat{\mu}}(\epsilon_*)$,
which by construction is at least $\exp(C \epsilon_*^{-2})$.

Let $((\Psi_n^t)_t)_n$ be a sequence of flows on the annulus $\A$ converging to $(\Psi^t)_t$ and such that, for each $n\ge 1$, the flow $(\Psi_n^t)_t$ satisfies:
\begin{itemize}
	\item for every $i \in \{0,1,\dots,n\}$, the curve $\cC_{i/n}$ is invariant and exponentially repelling; 
	\item for every $i \in \{0,1,\dots,n-1\}$, the curve $\cC_{(i+.5)/n}$ is invariant and exponentially attracting, with basin $h(\mathbb T\times (i/n , (i+1)/n))$; 
\end{itemize}
Let $f_n \coloneqq \Psi_n^{t_n}$, where $(t_n)$ is a sequence of non-zero numbers tending to zero.
Then $f_n$ converges $f = \id$ in the $C^\infty$ topology.
Tweaking the sequence $(t_n)$ if necessary, we can assume that each $f_n$ acts as an irrational rotation on each attracting cycle $\cC_{(i+.5)/n}$, $i \in \{0,1,\dots,n-1\}$.
Then every point in the basin of $\cC_{(i+.5)/n}$ has a well-defined empirical measure with respect to $f_n$, which is $\tilde \lambda_{(i+.5)/n}$.
Each of these basins has Lebesgue measure $1/n$, so the measure $\mathbf{e}^{f_n}_*(\Leb)$,
(which with some abuse of terminology we will call the ergodic decomposition of~$f_n$) 
is well defined and equals $\hat{\mu}_n$.
So for large enough $n$, 
the diffeomorphism $f_n$ displays high emergence at scale $\epsilon_*/2$:
$$
\Eme_{\Leb}(f_n)(\epsilon_*/2) = 
Q_{\hat{\mu}_n}(\epsilon_*/2) \ge \exp(C \epsilon_*^{-2}).
$$
(Strictly speaking, \cref{p:emergence_Q} does not apply since $\Leb$ measure is not $f_n$-invariant, but it still works since the empirical measures are $\Leb$-a.e.\ well defined and ergodic.)
For the remainder of the proof, we fix a large $n$ such that $f_n$ has the above properties, and moreover belongs to the given neighborhood $\cU$ of $f=\id$.

Now, if 
$g$ is a small $C^1$-perturbation of $f_n$ then by persistence of normally contracting submanifolds (see e.g.\ \cite[Thm.~2.1]{BB}), 
$g$ has $n$ attracting curves
$C^1$-close to the curves $\cC_{(i+.5)/n}$, and {their basins are} bounded by repelling curves
that are $C^1$-close to the curves $\cC_{i/n}$.
The rotation numbers along these attracting curves
are either irrational or rational with a large denominator, so every point in the union of the basins has a well-defined empirical measure with respect to $g$, which is close to $\tilde \lambda_{(i+.5)/n}$.
Thus $g$ has a well-defined ergodic decomposition, which is close to $\hat{\mu}_n$.
It follows from \cref{l:continuity} that:
$$
\Eme_{\Leb}(g)(\epsilon_*/4) = Q_{\mathbf{e}^g_*(\Leb)}(\epsilon_*/{ 4}
) \ge
Q_{\hat{\mu}_n}(\epsilon_*/2) \ge \exp(C \epsilon_*^{-2}).
$$
So it follows from the definition of $\epsilon_*$ that:
$$
\epsilon \coloneqq \frac{\epsilon_*}{4}
\quad \Rightarrow \quad \frac{\log \log \Eme_\Leb(g)(\epsilon)}{\log \epsilon} \ge 2-\epsilon_0 \, .
$$
Letting $\cV$ be a $C^1$-neighborhood of the diffeomorphism $f_n$ where such estimates hold, we conclude the proof of the \lcnamecref{forD} in the case $M = \A$, $f = \id$.

If $f$ is an arbitrary surface diffeomorphism admitting a periodic spot $O$, then we embed an annulus in $O$ and reproduce the construction above. 
Emergences can be estimated from below similarly. 
Details are left for the reader. 
\end{proof}

\begin{proof}[Proof of \cref{coroD}] 
The proof is entirely analogous to the proof of \cref{coroC}, using \cref{dense_periodic_spotD} instead of  \cref{dense_periodic_spotC} and \cref{forD} instead of \cref{forC}.
\end{proof}

\begin{appendix}
\section{Entropy}\label{s:entropy}

\subsection{Entropy in terms of covering numbers}\label{ss:Katok}

Let us explain how entropies are related to covering numbers. 
We  use these relations in \cref{s:top_em_examples}.

Let $f \colon X \to X$ be a continuous self-map of a compact metric space $(X,\mathsf{d})$.
For each integer $n \ge 1$, define the \emph{Bowen metric}:
\begin{equation}\label{e.Bowen}
\mathsf{d}_n(x,y) \coloneqq \max_{0 \le i < n} \mathsf{d}(f^i(x),f^i(y)) \, .
\end{equation} 
Let $N(n,\epsilon)  \coloneqq D_{\mathsf{d}_n}(\epsilon)$ denote the least number of balls of radii $\epsilon$ in the $\mathsf{d}_n$-metric necessary to cover $X$.
We recall the following:

\begin{definition}\label{def:topo_entropy} The \emph{topological entropy} of $f$ is:
\[h_\mathrm{top}(f) \coloneqq \lim_{\epsilon\to 0}\lim_{n\to \infty} \frac1n \log N(n, \epsilon).\]
\end{definition}

Fix an invariant measure $\mu \in \cM_f(X)$.
Given $n\ge 1$, $\epsilon > 0$, and $0 < \delta < 1$, let $N_\mu(n,\epsilon,\delta)$ denote the least number of balls of radii $\epsilon$ in the $\mathsf{d}_n$-metric necessary to cover a set of $\mu$-measure at least $1-\delta$.

Though metric entropy is most commonly defined in terms of measurable partitions, the following result by Katok allows us to define it in terms of covering numbers:

\begin{otherthm}[Katok~\cite{Katok}, Theorem~I.I]\label{t:Katok}
If $\mu$ is ergodic then for every $\delta$ in the range $0<\delta<1$,
\begin{equation}\label{e.Katok}
h_\mu(f) 
= \lim_{\epsilon \to 0} \liminf_{n\to \infty} \frac{1}{n} \log N_\mu(n,\epsilon,\delta)
= \lim_{\epsilon \to 0} \limsup_{n\to \infty} \frac{1}{n} \log N_\mu(n,\epsilon,\delta) \, .
\end{equation}
\end{otherthm}

When $f$ is a homeomorphism,  let $\tilde N_\mu(n,\epsilon,\delta)$ denote the least number of balls necessary to cover a set of $\mu$-measure at least $1-\delta$ of radii $\epsilon$ in the following metric:
\begin{equation}\label{e.Bowen.inv}
\tilde{\mathsf{d}}_n(x,y) \coloneqq \max_{-n < i < n} \mathsf{d}(f^i(x),f^i(y)) \, .
\end{equation}
We note that $\mathsf{d}_{2n}(f^{-n}(x), f^{-n}(y)) = \tilde{\mathsf{d}}_n (x,y)$ and so $\tilde N_\mu(n,\epsilon,\delta)=  N_\mu(2n,\epsilon,\delta)$.  
So we obtain: 

\begin{corollary}\label{c.Katok}
If $\mu$ is ergodic then for every $\delta$ in the range $0<\delta<1$,
\begin{equation}\label{e.Katok.inv}
h_\mu(f) 
= \lim_{\epsilon \to 0} \liminf_{n\to \infty} \frac{1}{2n} \log \tilde N_\mu(n,\epsilon,\delta)
= \lim_{\epsilon \to 0} \limsup_{n\to \infty} \frac{1}{2n} \log \tilde N_\mu(n,\epsilon,\delta) \, .
\end{equation}
\end{corollary}

\subsection{Variational principle for entropy}
If a measurable self-map $f$ of a measurable space $X$ preserves a probability measure $\mu$, then $h_\mu(f)$ denotes the corresponding metric entropy.
\begin{otherthm}[Variational Principle for Entropy]\label{them_variat_principle_entropy}
If $X$ is compact and $f$ is continuous, then the topological entropy $h_\mathrm{top}(f)$ equals the supremum of $h_\mu(f)$ where $\mu$ runs over all the invariant Borel probability measures. 
\end{otherthm}
Details can be found in the standard textbooks \cite{DGS,Mane,KH,PU,VO}.


\subsection{Metric entropy in terms of quantization numbers}\label{ss:entropy_quantization}

Let $(X,\mathsf{d})$ be a compact metric space, and let $\mathsf{W}_p$ and $\mathsf{LP}$ denote the induced Wasserstein and L\'evy--Prokhorov metrics on the space $\cM(X)$.
If $\mu \in \cM(X)$, then let $Q_{\mu, \mathsf{W}_p}(\mathord{\cdot})$ and $Q_{\mu, \mathsf{LP}}(\mathord{\cdot})$ and denote the corresponding quantization numbers.
They can be compared as follows:

\begin{lemma}\label{l:quant_comparison}
For every $\epsilon>0$,
$$
Q_{\mu, \mathsf{LP}}\left( \epsilon^{\frac{p}{p+1}} \right) \le Q_{\mu, \mathsf{W}_p}(\epsilon) \le
Q_{\mu, \mathsf{LP}}\left(\frac{\epsilon^p}{1+(\diam X)^p}\right) \, .
$$
\end{lemma}

\begin{proof}
This is an immediate consequence of inequalities \eqref{compaWpLP}.
\end{proof}


Given a continuous map $f \colon X \to X$ on the compact metric space $(X,\mathsf{d})$ and an integer $n \ge 1$,
the corresponding Bowen metric $\mathsf{d}_n$ induces Wasserstein and  L\'evy--Prokhorov metrics on the space $\cM(X)$, which we respectively denote by $\mathsf{W}_{p,n}$ and $\mathsf{LP}_n$.
Now, given an invariant measure $\mu \in \cM_f(X)$, we consider its quantization numbers with respect these two metrics.
This relates to the entropy as follows:

\begin{otherthm}[Reformulation of Katok's entropy theorem]\label{t:Katok_reform}
If $\mu$ is ergodic then:
\begin{alignat*}{2}
h_\mu(f) 
 &= \lim_{\epsilon \to 0} \liminf_{n\to \infty} \frac{1}{n} Q_{\mu, \mathsf{LP}_n}(\epsilon) 
&&= \lim_{\epsilon \to 0} \limsup_{n\to \infty} \frac{1}{n} Q_{\mu, \mathsf{LP}_n}(\epsilon)  \\
 &= \lim_{\epsilon \to 0} \liminf_{n\to \infty} \frac{1}{n} Q_{\mu, \mathsf{W}_{p,n}}(\epsilon) 
&&= \lim_{\epsilon \to 0} \limsup_{n\to \infty} \frac{1}{n} Q_{\mu, \mathsf{W}_{p,n}}(\epsilon)  \, .
\end{alignat*}
\end{otherthm}

\begin{proof}
Note that existence of limits as $\epsilon \to 0$ is automatic by monotonicity. 

In view of \cref{l:quant_comparison}, it is sufficient to consider the L\'evy--Prokhorov metrics.
By \cref{l:quant_comparison}, $Q_{\mu, \mathsf{LP}_n}(\epsilon) = N_\mu(n,\epsilon,\epsilon)$ (in the notation of \cref{ss:Katok}).

In the paper \cite{Katok} (see inequality (I.I)), Katok proves that:
$$
\forall \epsilon>0, \ \forall \delta>0, \quad {
\limsup_{n\to \infty}} \frac{1}{n} \log N(n,\epsilon,\delta) \le h_\mu(f) \, .
$$
(This is actually the ``easy part'' of the proof of \cref{t:Katok}, and a simple consequence of Shannon--MacMillan--Breiman's theorem.)
Taking $\delta = \epsilon$ and then taking $\epsilon \to 0$, we obtain:
$$
\lim_{\epsilon \to 0} {
\limsup_{n\to \infty}} \frac{1}{n} \log N(n,\epsilon,\epsilon) \le h_\mu(f)  \, .
$$

On the other hand, if $0 < \epsilon \le \delta < 1$ then $N(n,\epsilon,\epsilon) \ge N(n,\epsilon,\delta)$, so \cref{t:Katok} implies that:
$$
\lim_{\epsilon \to 0} \limsup_{n\to \infty} \frac{1}{n} \log N(n,\epsilon,\epsilon) \ge h_\mu(f)  \, .
$$
This concludes the proof.
\end{proof}

The reader will notice  a certain parallelism between the notions of topological/metric entropies  and topological/metric emergences: compare \cref{def:topo_entropy} with \cref{def.metric_em},  \cref{t:Katok_reform} with \cref{p:emergence_Q}, and \cref{them_variat_principle_entropy} with \cref{t:var_princ}. 
\end{appendix}



\begin{thebibliography}{DGMR} 

	

\MRbibitem[AB]{2888232}{AB}
\textsc{Avila, A.; Bochi, J.} -- 
Nonuniform hyperbolicity, global dominated splittings and generic properties of volume-preserving diffeomorphisms. 
\textit{Trans.\ Amer.\ Math.\ Soc.\ }364 (2012), no.\ 6, 2883--2907.
\doi{10.1090/S0002-9947-2012-05423-7}


\MRbibitem[ACW]{4198639}{ACW}
\textsc{Avila, A.; Crovisier, S.; Wilkinson, A.} --
$C^1$ density of stable ergodicity.
\textit{Adv.\ Math.\ }379 (2021), Article ID 107496, 69 p.
\doi{10.1016/j.aim.2020.107496}

\MRbibitem[BCLR]{2339286}{BCLR}
\textsc{B\'eguin, F.; Crovisier, S.; Le Roux, F.} --
Construction of curious minimal uniquely ergodic homeomorphisms on manifolds: the Denjoy-Rees technique.
\textit{Ann.\ Sci.\ \'Ecole Norm.\ Sup.\ }40 (2007), no.\ 2, 251--308. 
\doi{10.1016/j.ansens.2007.01.001}

\MRbibitem[Be1]{3514960}{Be1}
\textsc{Berger, P.} --
Generic family with robustly infinitely many sinks.
\textit{Invent.\ Math.\ }205 (2016), no.\ 1, 121--172.
\doi{10.1007/s00222-015-0632-6}

\MRbibitem[Be2]{3695404}{Berger}
\textsc{Berger, P.} -- 
Emergence and non-typicality of the finiteness of the attractors in many topologies.
\textit{Proc.\ Steklov Inst.\ Math.\ }297 (2017), no.\ 1, 1--27.
\doi{10.1134/S0081543817040010}

\MRbibitem[BeB]{3170632}{BB}
\textsc{Berger, P.; Bounemoura, A.} -- 
A geometrical proof of the persistence of normally hyperbolic submanifolds. \textit{Dyn. Syst.} 28 (2013), no.\ 4, 567--581. 
\doi{10.1080/14689367.2013.835386}



\MRbibitem[BeT]{3946863}{berger-turaev}
\textsc{Berger, P.; Turaev, D.} --
On Herman's positive entropy conjecture.
\textit{Adv.\ Math.\ }349 (2019), 1234--1288.
\doi{10.1016/j.aim.2019.04.002}


\MRbibitem[Bie]{4043881}{Biebler}
\textsc{Biebler, S.} -- 
Newhouse phenomenon for automorphisms of low degree in $\mathbb{C}^3$. 
\textit{Adv.\ Math.\ }361 (2020), Article ID 106952, 39 pp.
\doi{10.1016/j.aim.2019.106952}



\MRbibitem[Bil]{1700749}{Bill}
\textsc{Billingsley, P.} -- 
\textit{Convergence of probability measures.}
2nd edition. 
Wiley Series in Probability and Statistics.
John Wiley \& Sons, New York, 1999.

\bibitem[BiK]{BK}
\textsc{Birkhoff, G.D.; Koopman B.O.} -- 
Recent contributions to the ergodic theory. 
\textit{Proc.\ Natl.\ Acad.\ Sci.\ USA} 18 (1932), no.\ 3, 279--282.

\MRbibitem[BBG]{3819697}{BBG}
\textsc{Bochi, J.; Bonatti, C.; Gelfert, K.} -- 
Dominated Pesin theory: convex sum of hyperbolic measures.
\textit{Israel J.\ Math.\ }226 (2018), no.\ 1, 387--417. 
\doi{10.1007/s11856-018-1699-8}

\MRbibitem[BGV]{2280433}{BGV}
\textsc{Bolley, F.; Guillin, A.; Villani, C.} -- 
Quantitative concentration inequalities for empirical measures on non-compact spaces. 
\textit{Probab.\ Theory Related Fields} 137 (2007), no.\ 3--4, 541--593.
\doi{10.1007/s00440-006-0004-7}




\MRbibitem[BoD]{1670524}{BD}
\textsc{Bonatti, C.; D\'\i az, L.J.} --
Connexions h\'et\'eroclines et g\'en\'ericit\'e d'une infinit\'e de puits et de sources.
\textit{Ann.\ Sci.\ \'Ecole Norm.\ Sup.\ }32 (1999) no.\ 1,  135--150.
\doi{10.1016/S0012-9593(99)80012-3}


\MRbibitem[BoZ]{3968787}{BZ}
\textsc{Bonatti, C.; Zhang, J.} --  
Periodic measures and partially hyperbolic homoclinic classes. 
\textit{Trans.\ Amer.\ Math.\ Soc.\ }372 (2019), no.\ 2, 755--802.


\MRbibitem[BrS]{3292649}{BroerS}
\textsc{Broer, H.W.; Sevryuk, M.B.} --
KAM Theory: quasi-periodicity in dynamical systems, 
in \textit{Handbook of dynamical systems. Vol.\ 3,} 
H.W.\ Broer, B.\ Hasselblatt and F.\ Takens, Eds.,
North-Holland (Elsevier), Amsterdam, 2010,
Chapter 6, pp.\ 249--344.



\MRbibitem[BuW]{2630044}{BuW} 
\textsc{Burns, K.; Wilkinson, A.} -- 
On the ergodicity of partially hyperbolic systems. 
\textit{Ann.\ of Math.\ }171 (2010), no.\ 1, 451--489. 
\doi{10.4007/annals.2010.171.451}

\MRbibitem[Buz]{1441881}{Bu} 
\textsc{Buzzard, G.T.} -- 
 Infinitely many periodic attractors for holomorphic maps of {$2$} variables, 
 {\it Ann.\ of Math.\ }145,
(1997)
no.\ 2,  
 389--417. 
\doi{10.2307/2951819} 


\MRbibitem[CLN]{2376735}{heritage}
\textsc{Charpentier, E.; Lesne, A.; Nikolski, N.K.} (eds.) -- 
\textit{Kolmogorov's heritage in mathematics.}
Translated from the 2004 French original. 
Springer, Berlin, 2007.

\MRbibitem[DGS]{0457675}{DGS}
\textsc{Denker, M.; Grillenberger, C.; Sigmund, K.} -- 
\textit{Ergodic theory on compact spaces.}
Lecture Notes in Mathematics, Vol.\ 527. 
Springer-Verlag, Berlin-New York, 1976.


\MRbibitem[DGMR]{4018439}{DGMR}
\textsc{D\'\i az, L.J.; Gelfert, K.; Marcarini, T.; Rams, M.} --
The structure of the space of ergodic measures of transitive partially hyperbolic sets. 
\textit{Monatsh.\ Math.\ }190 (2019), no.\ 3, 441--479.
\doi{10.1007/s00605-019-01325-2}

\MRbibitem[DGR]{3936121}{DGR}
\textsc{D\'\i az, L.J.; Gelfert, K.; Rams, M.} --
Entropy spectrum of Lyapunov exponents for nonhyperbolic step skew-products and elliptic cocycles. 
\textit{Comm.\ Math.\ Phys.\ }367 (2019), no.\ 2, 351--416.
\doi{10.1007/s00220-019-03412-9}




\MRbibitem[DNP]{2267717}{DNP}
\textsc{D\'\i az, L.J.; Nogueira, A.; Pujals, E.R.} --
Heterodimensional tangencies.
\textit{Nonlinearity} 19 (2006), no.\ 11, 2543--2566. 
\doi{10.1088/0951-7715/19/11/003}


\MRbibitem[Dow]{1135237}{Dow}
\textsc{Downarowicz, T.} --
The Choquet simplex of invariant measures for minimal flows.
\textit{Israel J.\ Math. }74 (1991), no.\ 2-3, 241--256. 
\doi{10.1007/BF02775789}


\MRbibitem[Dua]{1287238}{Duarte} 
\textsc{Duarte, P.} -- Plenty of elliptic islands for the standard family of area preserving maps.
\textit{Ann.\ Inst.\ H.\ Poincar\'e Anal.\ Non Lin\'eaire} 11 (1994) no.4, 359--409. 
\doi{10.1016/S0294-1449(16)30180-9}


\MRbibitem[Dum]{322196}{Dumas}
\textsc{Dumas, H.S.} -- 
\textit{The KAM story.}
World Scientific, Hackensack, NJ, 2014.


\MRbibitem[Fa]{3236784}{Falconer}
\textsc{Falconer, K.} --
\textit{Fractal geometry: Mathematical foundations and applications.} 
3rd edition. John Wiley \& Sons, Chichester, 2014. 

\MRbibitem[GeR]{3820000}{GR}
\textsc{Gelfert, K.; Kwietniak, D.} --
On density of ergodic measures and generic points.
\textit{Ergodic Theory Dynam.\ Systems} 38 (2018), no.\ 5, 1745--1767. 
\doi{10.1017/etds.2016.97}


\MRbibitem[GeT]{2644327}{GT10}
\textsc{Gelfreich, V.; Turaev, D.} -- 
Universal dynamics in a neighborhood of a generic elliptic periodic point.
\textit{Regul.\ Chaotic Dyn.\ }15 (2010), no.\ 2--3, 159--164. 
\doi{10.1134/S156035471002005X}


\bibitem[GiS]{GiS}
\textsc{Gibbs, A.L.; Su, F.E.} -- 
On choosing and bounding probability metrics.
\textit{Internat.\ Statist.\ Rev.\ }70 (2002), no.\ 3, 419--435.
\doi{10.1111/j.1751-5823.2002.tb00178.x}

\MRbibitem[GTS]{2290462}{GTS07}
\textsc{Gonchenko, S.; Turaev, D.; Shilnikov, L.} -- 
Homoclinic tangencies of arbitrarily high orders in conservative and dissipative two-dimensional maps. 
\textit{Nonlinearity} 20 (2007), no.\ 2, 241--275. 
\doi{10.1088/0951-7715/20/2/002}

\MRbibitem[GoP]{3666070}{GoP}
\textsc{Gorodetski, A.; Pesin, Ya.} --
Path connectedness and entropy density of the space of hyperbolic ergodic measures.
\textit{Modern theory of dynamical systems}, 111--121,
\textit{Contemp.\ Math.}, 692, Amer.\ Math.\ Soc., Providence, RI, 2017. 


\MRbibitem[GrL]{1764176}{GrafL}
\textsc{Graf, S.; Luschgy, H.} -- 
\textit{Foundations of quantization for probability distributions.}
Lecture Notes in Mathematics, 1730. Springer-Verlag, Berlin, 2000.


\MRbibitem[GrS]{2059709}{GrimS}
\textsc{Grimmett, G.R.; Stirzaker, D.R.} --  
\textit{Probability and random processes.} 
3rd edition. Oxford University Press, New York, 2001. 

\MRbibitem[Ka]{0573822}{Katok}
\textsc{Katok, A.} -- 
Lyapunov exponents, entropy and periodic orbits for diffeomorphisms. 
\textit{Inst.\ Hautes \'Etudes Sci.\ Publ.\ Math.\ }51 (1980), 137--173.


\MRbibitem[KaH]{1326374}{KH}
\textsc{Katok, A.; Hasselblatt, B.} --
\textit{Introduction to the modern theory of dynamical systems.}
Encyclopedia of Mathematics and its Applications, 54. 
Cambridge University Press, Cambridge, 1995. 




\MRbibitem[Kl1]{2949240}{Klo12}
\textsc{Kloeckner, B.} --
A generalization of Hausdorff dimension applied to Hilbert cubes and Wasserstein spaces.
\textit{J.\ Topol.\ Anal.\ }4 (2012), no.\ 2, 203--235.
\doi{10.1142/S1793525312500094}

\MRbibitem[Kl2]{3333967}{Klo15}
\textsc{Kloeckner, B.} -- 
A geometric study of Wasserstein spaces: ultrametrics. 
\textit{Mathematika} 61 (2015), no.\ 1, 162--178. 
\doi{10.1112/S0025579314000059}

\MRbibitem[KoT]{0124720}{KT}
\textsc{Kolmogorov, A.N.; Tihomirov, V.M.} --
$\epsilon$-entropy and $\epsilon$-capacity of sets in functional space. 
\textit{Amer.\ Math.\ Soc.\ Transl.\ }(2) 17 (1961), 277--364.



\MRbibitem[KuZ]{1353511}{KZ}
\textsc{Kulkarni, S.R.; Zeitouni, O.} --
A general classification rule for probability measures.
\textit{Ann.\ Statist.\ }23 (1995), no.\ 4, 1393--1407. 
\doi{10.1214/aos/1176324714}


\MRbibitem[LiM]{1877974}{LM}
\textsc{Lindsay, L.J.; Mauldin, R.D.} -- 
Quantization dimension for conformal iterated function systems.
\textit{Nonlinearity} 15 (2002), no.\ 1, 189--199. 
\doi{10.1088/0951-7715/15/1/309}

\MRbibitem[MK]{2461035}{MacKay}
\textsc{MacKay, R.S.} -- 
Nonlinearity in complexity science.
\textit{Nonlinearity} 21 (2008), no.\ 12, T273--T281. 
\doi{10.1088/0951-7715/21/12/T03}

\MRbibitem[Ma1]{0678479}{Mane82}
\textsc{Ma\~{n}\'{e}, R.} --
An ergodic closing lemma. 
\textit{Ann.\ of Math.\ }116 (1982), no.\ 3, 503--540. 
\doi{10.2307/2007021}

\MRbibitem[Ma2]{0889254}{Mane}
\textsc{Ma\~{n}\'{e}, R.} --
\textit{Ergodic theory and differentiable dynamics.}
Translated from the Portuguese by Silvio Levy. 
Ergebnisse der Mathematik und ihrer Grenzgebiete, 8. Springer--Verlag, Berlin, 1987.

\MRbibitem[Mo]{0147741}{Moser}
\textsc{Moser, J.} -- 
On invariant curves of area-preserving mappings of an annulus. 
\textit{Nachr.\ Akad.\ Wiss.\ G\"ottingen Math.-Phys.\ Kl.\ II} 1962 (1962), 1--20. 


\MRbibitem[Ne1]{0339291}{Ne1}
\textsc{Newhouse, S.} --
Diffeomorphisms with infinitely many sinks.
\textit{Topology} 13 (1974), 9--18. 
\doi{10.1016/0040-9383(74)90034-2}

\MRbibitem[Ne2]{556584}{Ne2} 
\textsc{Newhouse, S.} --
The abundance of wild hyperbolic sets and nonsmooth stable sets for diffeomorphisms.
\textit{Inst.\ Hautes \'Etudes Sci.\ Publ.\ Math.\ }50 (1979), 101--151.

\MRbibitem[Ng]{3059422}{Nguyen}
\textsc{Nguyen, XuanLong.} -- 
Convergence of latent mixing measures in finite and infinite mixture models.
\textit{Ann.\ Statist.\ }41 (2013), no.\ 1, 370--400. 
\doi{10.1214/12-AOS1065}


\MRbibitem[Ob]{4069249}{Obata}
\textsc{Obata, D.} -- 
On the stable ergodicity of Berger--Carrasco's example.
\textit{Erg.\ Theory Dyn.\ Systems }40 (2020), no. 4, 1008--1056.
\doi{10.1017/etds.2018.65}

\MRbibitem[OxU]{0005803}{OxtobyUlam}
\textsc{Oxtoby, J.C.; Ulam, S.M.} -- 
Measure-preserving homeomorphisms and metrical transitivity. 
\textit{Ann.\ of Math.\ }42 (1941), 874--920. 
\doi{10.2307/1968772}

\MRbibitem[Pe]{1489237}{Pesin}
\textsc{Pesin, Ya.~B.} -- 
\textit{Dimension theory in dynamical systems: Contemporary views and applications.}
Chicago Lectures in Mathematics. 
University of Chicago Press, Chicago, IL, 1997. 




\MRbibitem[P\"o]{0668410}{Poschel}
\textsc{P\"oschel, J.}
Integrability of Hamiltonian systems on Cantor sets.
\textit{Comm.\ Pure Appl.\ Math.\ }35 (1982), no.\ 5, 653–696.
\doi{10.1002/cpa.3160350504}



\MRbibitem[PrU]{2656475}{PU}
\textsc{Przytycki, F.; Urba\'{n}ski, M.} -- 
\textit{Conformal fractals: ergodic theory methods.}
London Mathematical Society Lecture Note Series, 371. Cambridge University Press, Cambridge, 2010. 


\MRbibitem[Rog]{0172183}{Rogers}
\textsc{Rogers, C.A.} -- 
\textit{Packing and covering.}
Cambridge Tracts in Mathematics and Mathematical Physics, No.\ 54.
Cambridge University Press, New York, 1964.


\MRbibitem[Sig]{0447528}{Sigmund}
\textsc{Sigmund, K.} -- 
On the connectedness of ergodic systems.
\textit{Manuscripta Math.\ }22 (1977), no.\ 1, 27--32. 
\doi{10.1007/BF01182064}


\MRbibitem[Tu]{3320312}{Turaev}
\textsc{Turaev, D.} --
Maps close to identity and universal maps in the Newhouse domain.
\textit{Comm.\ Math.\ Phys.\ }335 (2015), no.\ 3, 1235--1277. 
\doi{10.1007/s00220-015-2338-4}


\MRbibitem[ViO]{3558990}{VO}
\textsc{Viana, M.; Oliveira, K.} -- 
\textit{Foundations of ergodic theory.}
Cambridge Studies in Advanced Mathematics, 151. 
Cambridge University Press, Cambridge, 2016.



\MRbibitem[Vil]{1964483}{Villani}
\textsc{Villani, C.} --
\textit{Topics in optimal transportation.}
 Graduate Studies in Mathematics, 58. American Mathematical Society, Providence, RI, 2003.


\end{thebibliography}
\end{document}